\documentclass[11pt,a4paper,leqno]{article}
\usepackage[textwidth=16cm, textheight=23cm,marginratio=1:1]{geometry}
\usepackage[utf8]{inputenc}
\usepackage[T1]{fontenc}
\usepackage{amsmath}
\usepackage{amsfonts}
\usepackage{amssymb}
\usepackage{tikz-cd}
\usepackage{amsthm}
\usepackage{booktabs} 
\usepackage[textwidth=0.9in]{todonotes}

\usepackage{caption} 
\usepackage{enumitem}
\usepackage[breaklinks,bookmarks=false, hidelinks]{hyperref} 
\hypersetup{
	colorlinks,
	linkcolor={red}
}
\numberwithin{equation}{section}

\newtheorem{theorem}[equation]{Theorem}
\newtheorem{lemma}[equation]{Lemma}
\newtheorem{proposition}[equation]{Proposition}
\newtheorem{corollary}[equation]{Corollary}
\newtheorem*{theorem*}{Theorem}

\theoremstyle{definition}
\newtheorem{definition}[equation]{Definition}
\newtheorem{remark}[equation]{Remark}

\newtheorem*{notation}{Notation}
\newtheorem*{question*}{Question}

\theoremstyle{remark}
\newtheorem{example}[equation]{Example}

\DeclareMathOperator{\rank}{rank}

\DeclareMathOperator{\inn}{in}

\DeclareMathOperator{\gr}{gr}

\DeclareMathOperator{\Sym}{Sym}

\DeclareMathOperator{\rord}{rord}

\DeclareMathOperator{\Soc}{Soc}

\DeclareMathOperator{\Gr}{Gr}

\DeclareMathOperator{\GL}{GL}

\DeclareMathOperator{\Spec}{Spec}

\DeclareMathOperator{\Ann}{Ann}
\DeclareMathOperator{\Apolar}{Apolar}
\DeclareMathOperator{\Hilb}{Hilb}
\DeclareMathOperator{\Diff}{Diff}

\newcommand{\hatR}{\hat{R}}
\newcommand{\frakp}{\mathfrak{p}}
\newcommand{\mm}{\mathfrak{m}}

\newcommand{\cU}{\mathcal{U}}
\newcommand{\PP}{\mathbb{P}}
\newcommand{\KK}{\mathbb{K}}

\renewcommand{\AA}{\mathbb{A}}

\newcommand{\Gmult}{\mathbb{G}_{\mathrm{m}}}%
\newcommand{\HilbFunc}[2]{\Hilb_{#1}(\AA^{#2},0)}

\newcommand{\HilbFuncSm}[2]{\Hilb_{#1}^{\mathrm{sm}}(\AA^{#2})}

\newcommand{\HilbFuncGr}[2]{\Hilb^{ \Gmult}_{#1}(\AA^{#2},0)}
\newcommand{\hook}{\:\lrcorner\:}
\newcommand{\CommMat}[2]{C_{#2}(\mathbb{M}_{#1})}
\newcommand{\cA}{\mathcal{A}}

\newcommand{\X}[3]{\mathcal{X}^{#1}_{1,#2,#1,1;#3}}

\begin{document}
\title{Components of Hilbert Schemes of low degree and smoothable algebras}
\author{Maciej Ga{\l}\k{a}zka\thanks{Daegu-Gyeongbuk Institute of Science and Technology, 333 Technojungang-daero, Hyeonpung-eup, Dalseong-gun, Daegu, 42988, Republic of Korea, and Warsaw University of Life Sciences, Nowoursynowska 166, 02-787 Warszawa, Poland},
Hanieh Keneshlou\thanks{Department of Mathematics and Statistics, Universitätsstraße 10, 78464 Konstanz, Germany.},
Klemen Šivic\thanks{Faculty of mathematics and physics, University of Ljubljana, Jadranska 19, 1000 Ljubljana, Slovenia, and Institute of mathematics, physics and mechanics, Jadranska 19, 1000 Ljubljana, Slovenia}}
\maketitle
\begin{abstract}
    In this article, we describe the irreducible components of the Hilbert scheme of $d$ points on $\mathbb{A}^n$ for $d=9,10$. The main techniques we use are the variety of commuting matrices and analyzing loci of local algebras with a specific Hilbert function. We further prove that any finite local algebra of degrees $9,10$ and the socle dimension $2$ is smoothable. As the main consequence, we establish the equality of the cactus Grassmann and the secant Grassmann variety in the corresponding cases.
\end{abstract}
\tableofcontents
\section*{Introduction}
Hilbert schemes are fundamental objects in algebraic geometry parametrizing families of ideals in a polynomial ring. Since being introduced by Grothendieck, they are in the center of study in algebraic geometry, and they have paved the way for the construction of most moduli spaces.

Among others, the Hilbert scheme of points $\Hilb_d(\mathbb{A}^n) $ parametrizing zero-dimensional subschemes of $\AA^n$ of degree $d$, occupies a special place and has found connection to various research lines such as combinatorics \cite{Mar, Mar2}, enumerative geometry \cite{enu1,enu2}, string theory \cite{string} and to motivic homotopy theory \cite{Motiv}. It has a rich geometry to be still understood.
Although $\Hilb_d(\mathbb{A}^n)$ is smooth and irreducible for $n\leq 2$, very little is known about the components of $\Hilb_d(\AA^n)$ as $n$ grows. There is a distinguished component $\Hilb_d^{sm}(\mathbb{A}^n)$, the so-called smoothable component, whose general points parametrize $d$-tuples of distinct points. This is a generically smooth component of dimension $nd$, yet the Hilbert scheme of points can have components of excess dimension.

Following \cite{Iarrobino73}, a component is called elementary if it parametrizes subschemes supported at a single point. Elementary components are building blocks of all components of $\Hilb_d(\mathbb{A}^n)$, see \cite[Lemma 1]{Iarrobino73}. By results in \cite{ccvv}, it is known that $\Hilb_d(\mathbb{A}^n)$ is irreducible for $d\leq 7$, and $\Hilb_8(\mathbb{A}^n)$ is union of the smoothable and an elementary component for $n\geq 4$.

In this article, we determine the elementary components of the Hilbert scheme of points for $d=9,10$ and $n\geq 4$. As the main result, we prove the following theorem.
\begin{theorem}
  The only elementary components of $\Hilb_9(\AA^n)$ and $\Hilb_{10}(\AA^n)$ are the components parameterizing algebras with Hilbert function $(1,5,3)$ and $(1,6,3)$, respectively. 
\end{theorem}
This provides an answer to question VIII of \cite{openproblem}. See Subsection~\ref{subsection:mainTh} for a proof. In the following table, we summarize the data of the components of the Hilbert scheme for $8\leq d\leq 10$. We assume $n\geq 4$, as otherwise the Hilbert scheme is irreducible.
\vspace*{0.5cm}
\begin{small}
    \begin{center}
    \begin{tabular}{c c c c}\label{table1}
            \hfill 
            &  $d = 8$ & $d  = 9 $  & $d = 10$
            \\\toprule 
            $n = 4$ & $T_{sm},\ T_{(1,4,3)}^4$ & $T_{sm},\ T_{(1,4,3)\bullet}^4$  & $T_{sm}, \ T_{(1,4,3)\bullet\bullet}^4$  \\
            $n = 5$ & $T_{sm},\ T_{(1,4,3)}^5$ &$ T_{sm}, \ T_{(1,4,3)\bullet}^5,\ T_{(1,5,3)}^5$ & $T_{sm},\ T_{(1,4,3)\bullet\bullet}^5,\ T_{(1,5,3)\bullet}^5$ \\
            $n \ge 6$ & $T_{sm},\ T_{(1,4,3)}^n$ &$ T_{sm},\ T_{(1,4,3)\bullet}^n,\ T_{(1,5,3)}^n$ & $T_{sm}, \ T_{(1,4,3)\bullet\bullet}^n,\  T_{(1,5,3)\bullet}^n,\ T_{(1,6,3)}^n$  
        \end{tabular}\\\vspace{3mm}
            {Table~1. \small The components of $\Hilb_d(\mathbb{A}^n)$}
    \end{center}
\end{small}

In Table~1, $T_{sm}$ denotes the smoothable component. The symbol $T^{n}_H$ denotes an elementary component of $\Hilb(\AA^n)$, whose general point is the spectrum of a local algebra with Hilbert function $H$. The symbol $T^{n}_{H\bullet}$ denotes a non-elementary component of $\Hilb(\AA^n)$, whose general point is a union of the spectrum of a local algebra with Hilbert function $H$, and a point, and so on.

In Table~2, for every component, we give a point on it in the language of ideals. The ideals corresponding to $T^4_{(1,4,3)}, T^5_{(1,5,3)}, T^6_{(1,6,3)}$ are the main building blocks. See Subsections~\ref{subsection_socle} and \ref{Macaulay's Inverse System} for the explanation of the symbol $\Ann$.

\begin{small}
    \begin{center}
    \begin{tabular}{c c c c}
            \hfill 
              \text{component} & \text{ideal} \\
            \toprule 
              $T_{(1,4,3)}^4$ & $\Ann(x_1 x_2, x_3 x_4, x_1 x_3 + x_2 x_4)$ \\
              $T_{(1,4,3)}^5$ & $\Ann(x_1 x_2, x_3 x_4, x_1 x_3 + x_2 x_4,x_5)$ \\
              $T_{(1,4,3)}^6$ & $\Ann(x_1 x_2, x_3 x_4, x_1 x_3 + x_2 x_4,x_5,x_6)$ \\
              $T_{(1,4,3)\bullet}^4$ & $\Ann(x_1 x_2, x_3 x_4, x_1 x_3 + x_2 x_4) \cap (\alpha_1 -\lambda_1,\ldots,\alpha_4 - \lambda_4)$ \\
              $T_{(1,4,3)\bullet}^5$ & $\Ann(x_1 x_2, x_3 x_4, x_1 x_3 + x_2 x_4,x_5)\cap (\alpha_1 - \lambda_5,\ldots,\alpha_5-\lambda_5)$  \\
              $T_{(1,4,3)\bullet}^6$ & $\Ann(x_1 x_2, x_3 x_4, x_1 x_3 + x_2 x_4,x_5,x_6)\cap (\alpha_1-\lambda_1,\ldots, \alpha_6 - \lambda_6)$  \\
              $T_{(1,4,3)\bullet\bullet}^4$ & $\Ann(x_1 x_2, x_3 x_4, x_1 x_3 + x_2 x_4) \cap (\alpha_1 -\lambda_1,\ldots,\alpha_4 - \lambda_4)\cap(\alpha_1-\mu_1,\ldots,\alpha_4 - \mu_4)$ \\
              $T_{(1,4,3)\bullet\bullet}^5$ & $\Ann(x_1 x_2, x_3 x_4, x_1 x_3 + x_2 x_4,x_5)\cap (\alpha_1 - \lambda_5,\ldots,\alpha_5-\lambda_5)\cap(\alpha_1-\mu_1,\ldots,\alpha_5-\mu_5)$  \\
              $T_{(1,4,3)\bullet\bullet}^6$ & $\Ann(x_1 x_2, x_3 x_4, x_1 x_3 + x_2 x_4,x_5,x_6)\cap (\alpha_1-\lambda_1,\ldots, \alpha_6 - \lambda_6)\cap(\alpha_1 - \mu_1,\ldots,\alpha_6 - \mu_6) $  \\
              $T^5_{(1,5,3)}$ & $\Ann(-\frac{1}{2}x_1^2 + x_1x_4 + x_2x_5, x_1x_3 +  x_3x_5, x_1x_2 + x_2x_3 +  x_4x_5)$ \\
              $T^6_{(1,5,3)}$ & $\Ann(-\frac{1}{2}x_1^2 + x_1x_4 + x_2x_5, x_1x_3 +  x_3x_5, x_1x_2 + x_2x_3 +  x_4x_5,x_6)$ \\
              $T^5_{(1,5,3)\bullet}$ & $\Ann(-\frac{1}{2}x_1^2 + x_1x_4 + x_2x_5, x_1x_3 +  x_3x_5, x_1x_2 + x_2x_3 +  x_4x_5)\cap(\alpha_1-\lambda_1,\ldots,\alpha_5 - \lambda_5)$ \\
              $T^6_{(1,5,3)\bullet}$ & $\Ann(-\frac{1}{2}x_1^2 + x_1x_4 + x_2x_5, x_1x_3 +  x_3x_5, x_1x_2 + x_2x_3 +  x_4x_5,x_6)\cap(\alpha_1-\lambda_1,\ldots,\alpha_6 - \lambda_6)$ \\
              $T^6_{(1,6,3)}$ & $\Ann(x_1x_4 + x_2x_5 +x_3 x_6, x_1x_3 + x_2 x_4 + x_3 x_5 + x_4 x_6, x_1 x_2 + x_2 x_3 + x_3 x_4 + x_4 x_5 + x_5 x_6)$
              
        \end{tabular} \\\vspace{3mm}
        {Table~2. \small Example of ideal of a point on each component in Table 1}
    \end{center}
\end{small}

We exploit the counterpart of the Hilbert scheme, the so-called variety of commuting matrices and the correspondence between their irreducible components worked out in \cite{JS} to show that the $n$-tuple $(A_1,\ldots,A_n)$ of $d\times d$ commuting matrices corresponding to each local algebra of a Hilbert function different from $(1,5,3)$ and $(1,6,3)$ comes as degeneration of tuples lying on a non-elementary component. 

In connection of the Hilbert scheme of points to so-called Grassmann secant and cactus varieties, parameterizing tensors and polynomials of bounded border rank, we analyze further the open subscheme of the Hilbert scheme, parameterizing algebras of socle dimension at most $2$, showing that this subscheme is inside the smoothable component. As a crucial consequence, we establish the irreducibility of Grassmann cactus variety of pencils, extending results in \cite{gmr}.

The organization of this article is as follows. In Section \ref{Pre} we collect important algebraic ingredients such as socle dimension, Macaulay's inverse system, and the general known results on Hilbert scheme of points and the variety of commuting matrices. Section \ref{DescriptionOfComponents} is devoted to determining the components of $\HilbFunc{H}{n}$ parameterizing local algebras with Hilbert functions $H=(1,n,3,2)$ and $H=(1,n,2,2,1)$, using the description of the components of the corresponding multigraded Hilbert schemes. This will be later used in Section \ref{Local1432}. In Section \ref{RayMethod}, exploiting the ray degeneration method, we prove that algebras with Hilbert functions of the form $(1,H_1,\ldots,H_c,1,\ldots, 1)$ and socle degree $ s\geq 2c$ are in a non-elementary component. Moreover, all those algebras of socle dimension $2$ lie on the smoothable component.
In Section \ref{Components}, we establish that all the components of $\Hilb_9(\mathbb{A}^n)$ and $\Hilb_{10}(\mathbb{A}^n)$ other than the components parameterizing algebras with Hilbert functions $(1,5,3)$ and $(1,6,3)$ are non-elementary, in particular the treated algebras of socle dimension $2$ are smoothable.  As an important follow-up of our results, in Section \ref{GrasCactus} we demonstrate that the Grassmann cactus variety of pencils in degree $9$ and $10$ is irreducible.

\section*{Acknowledgment}
We would like to thank Joachim Jelisiejew for valuable comments and remarks, and reminding us of the linkage technique elaborated in Remark \ref{linkage}. We are grateful to Tomasz Ma\'ndziuk whose comments improved the presentation of our article. We also thank Michele Graffeo and Joseph Landsberg for informing us about the exception $(1,5,4)$, remarked in \ref{exception}. The first author is supported by the project ``Global Basic Research Laboratory: Algebra and Geometry of Spaces of Tensors, and Applications'', RS-2024-00414849, awarded by the NRF of Korea, and also by the European Union under NextGenerationEU. PRIN 2022 Prot. n. 2022ZRRL4C\_004. The third author is partially supported by the Slovenian Research and Innovation Agency program P1-0222 and grants J1-3004, J1-60011 and J1-50001.

Views and opinions expressed are however those of the authors only and do not necessarily reflect those of the European Union or European Commission. Neither the European Union nor the granting authority can be held responsible for them.
\subsection*{Notation}
Here we summarize all the loci appearing throughout this article:\\

\begin{tabular}{c c}
    object & symbol\\
    \toprule
    The Hilbert scheme of $d$ points & $\Hilb_d(\mathbb{A}^n)$\\
    The smoothable component & $\Hilb_d^{sm}(\mathbb{A}^n)$\\
    The punctual Hilbert scheme & $\Hilb_d(\mathbb{A}^n, 0)$\\
    The locus of algebras with socle dimension $\leq \tau$ & $\Hilb_d^{\leq \tau}(\mathbb{A}^n)$\\
    The locus of homogeneous ideals & $\Hilb_d^{\Gmult}(\mathbb{A}^n,0)$\\
    The locus of homogeneous ideals with fixed Hilbert function $H$ &
    $\Hilb_H^{\Gmult}(\mathbb{A}^n,0)$\\
    The locus with fixed Hilbert function $H$ & $\HilbFunc{H}{n}$\\
    The locus of Gorenstein algebras &
    $\Hilb_d^{Gor}(\mathbb{A}^n)$
\end{tabular}

\section{Preliminaries}\label{Pre}
In this section, we collect facts and notation regarding the
Hilbert scheme, variety of commuting matrices, the Hilbert function, Macaulay's inverse systems and ray families. Throughout this article, we work over an
algebraically closed field $\KK$ of characteristic zero. 
\subsection{Socle and Hilbert function}\label{subsection_socle}
Let $\hat{R}=\mathbb{K}[[\alpha_1,\ldots,\alpha_n]]$ be the power series ring with the maximal ideal $\mathfrak{m}=(\alpha_1,\ldots,\alpha_n)$. Let $\cA=\hat{R}/I$ be a local Artin ring for an $\mm$-primary ideal $I\subset \hat{R}$,
with the unique maximal ideal $\mathfrak{n}=\mm/I$ and residue field $\mathbb{K}$.
Recall that an \textit{$\mm$-primary ideal} is just an ideal $I$ such that
$I\supset \mm^r$ for some $r$. Let $R= \KK[\alpha_1, \ldots ,\alpha_n]\subset \hat{R}$ be the polynomial ring. Since for every $r$ we have
canonically
\begin{equation}\label{eq:ringchange}
    \frac{\hat{R}}{\mm^r}  \simeq
    \frac{R}{(\alpha_1, \ldots,\alpha_n)^r},
\end{equation}
we could view $\mathcal{A}$ as a quotient of $R$ by an ideal $I$
satisfying $I\supset (\alpha_1, \ldots ,\alpha_n)^r$ for some $r$.
\textit{The socle} $\Soc(\cA)$ of $\cA$ is the annihilator of the maximal ideal in $\cA$.
The socle is a $\KK$-vector space. \textit{The socle dimension} of $\cA$ is
defined by $\tau(\cA):=\dim_{\KK}\Soc(\cA)$. The ring $\cA$ is
\textit{Gorenstein} if
$\tau(\cA)=1$, see ~\cite[Chapter~21]{EisView}. The \textit{associated graded ring of $\cA$}, denoted $\gr(\cA)$, is the vector space
$\bigoplus_{i\geq
0}\mathfrak{n}^i/\mathfrak{n}^{i+1}$ with natural multiplication. When $\cA
\simeq
\KK[\alpha_1, \ldots ,\alpha_n]/I$, where $I\supset (\alpha_1, \ldots ,\alpha_n)^r$, then $\gr(\cA)
\simeq \KK[\alpha_1, \ldots ,\alpha_n]/\inn(I)$, where $\inn(I)$ is the
ideal
generated by the smallest degree forms of elements of $I$.
We have
$\tau(\gr \cA) \geq \tau(\cA)$ and typically strict inequality occurs.
\begin{definition}
    The Hilbert function $H\colon\mathbb{N}\rightarrow \mathbb{N}$ of a local ring $\cA$
with a maximal ideal $\mathfrak{n}$ is defined by $H(i):=\dim_{\KK}
\mathfrak{n}^i/\mathfrak{n}^{i+1}$. It is the Hilbert function of its
associated graded ring.
\end{definition}
Let $s$ denote the largest integer such that $\mathfrak{n}^s\neq 0$, the
so-called \textit{socle degree} of $\cA$. The Hilbert function of $\cA$ can be
then represented by a vector $ H=(1,H(1),\ldots, H(s))$ or a series
$\sum_{i=0}^s H(i)T^i$. Moreover, since $\mathfrak{n}^{s}\subset \Soc(\cA)$, we
get $H(s)\leq \tau(\cA)$. We define the Hilbert function of an
$\cA$-module similarly.\\
Throughout this article, for an algebra with a given Hilbert function, we assume $H(1)$ is the embedding dimension of the corresponding scheme, so that the scheme is embedded in $\AA^{H(1)}$ and not in an affine space of higher dimension.
\subsection{Macaulay's Inverse System}\label{Macaulay's Inverse System}
Let $S=\KK[x_1,\ldots,x_n]$ be the polynomial ring with $n$ variables which we
view as a vector space.
Recall the rings $R = \KK[\alpha_1, \ldots ,\alpha_n]\subset \hat{R} = \KK[[\alpha_1, \ldots
,\alpha_n]]$.
The
ring $\hat{R}$ acts on $S$ by the partial derivation map $\lrcorner\colon\hat{R}\times S \rightarrow S$
defined as follows:
\[
    \alpha^\textbf{a}\lrcorner  x^\textbf{b}:=\begin{cases}
        \frac{\textbf{b}!}{(\textbf{b}-\textbf{a})!} x^{\textbf{b}-\textbf{a}} & \textbf{b}\geq  \textbf{a},\\
        0 & \mbox{otherwise},
    \end{cases}
\]
where $\textbf{a}=(a_1,\ldots,a_n)$ and $\textbf{b}=(b_1,\ldots,b_n)$
are vectors in $\mathbb{N}^n$, $\textbf{d}!  = \prod_{i=1}^n
d_i!$ for $\textbf{d} = (d_1,\ldots,d_n) \in \mathbb{N}^n$, and $\textbf{b}\geq \textbf{a}$ if  $b_i\geq
a_i$ for all $1\leq i\leq n$. With this action, $S$ can be viewed as an
$\hat{R}$-module. In fact, the partial derivation map is up to scalars equivalent to the
action of $\hat{R}$ on $S$ by contraction, see for
example~\cite[Appendix~A]{iakanev}, because we are over a field of
characteristic zero.
We will mostly work with the polynomial ring $R$ rather than the whole
$\hat{R}$, thus below we restrict to the $R$-action on $S$. This is mostly a
formal choice.
\begin{definition}
Let $I\subset R$ be an $(\alpha_1, \ldots ,\alpha_n)$-primary ideal of $R$. The
Macaulay \emph{inverse system of $I$} is
\[
    I^{\perp}=\{ g\in S: \ I\lrcorner g= 0\}.
\]
This is an $R$-submodule of $S$. Its generators are called
\emph{dual generators} of $I$.
Given a subset $E \subset
S$, the annihilator of $E$ is the ideal
\[
    \Ann_{R}(E)=\{f\in R: \ f\lrcorner E=0\},
\]
which is called the \textit{apolar ideal} of $E$.

We also define 
\begin{equation*}
    \Apolar_{R}(E) := R/\Ann_{R}(E).
\end{equation*}
\end{definition}

Note that $\Ann_{R}(E) = \Ann_{R}(M)$, where $M = R\lrcorner E$ is the
$R$-submodule
of $S$ generated by $E$. If $I\subseteq R$ is a homogeneous ideal, then
$I^{\perp}$ is spanned by homogeneous polynomials, and if $E$ is spanned by
homogeneous elements, then $\Ann_R(E)\subseteq R$ is homogeneous. Also, if $I\supset (\alpha_1, \ldots ,\alpha_n)^r
$,  then $I^{\perp} \subset S_{<r}$.
We will often abbreviate $\Ann_R(-)$ to $\Ann(-)$ when the ring is clear from
the context.\\

The above constructions are
justified by the following theorem due to Macaulay \cite{macauly}. 
\begin{theorem}[Macaulay's duality]
    For every $(\alpha_1, \ldots ,\alpha_n)$-primary ideal $I$ of $R$ and finitely generated
    $R$-submodule $M\subset S$ we have $\Ann_{R}(I^{\perp}) = I$ and
    $(\Ann_{R}(M))^{\perp} = M$. In this way, the operations $\perp$ and
    $\Ann_{R}(-)$ give
    an inclusion-reversing bijection between finitely generated $R$-submodules of
    $S$ and $(x_1, \ldots ,x_n)$-primary ideals of $R$. Moreover, the $R$-module
    $I^\perp$ is minimally generated by $\tau(R/I)$ elements.
    A zero-dimensional $\cA=R/I$ is Gorenstein of socle degree $s$ if and only
    if $I^\perp$ is a principal $R$-module generated by a polynomial of degree $s$.
\end{theorem}
\begin{notation}
We say that a graded algebra $\cA$ is of socle type $1^{a_1}2^{a_2}\ldots s^{a_s}$ if the inverse system of $\cA$ is minimally generated by $a_i$ generators in each degree $i$, for $i=1,\ldots, s$.
\end{notation}
\subsection{Ray families}\label{subsection_ray_family}
In this subsection, we mainly follow \cite{casnati_jelisiejew_notari}.
\begin{definition}\label{iray}
    Let $I\subset \hatR$ be an ideal of finite colength and $\pi_i:\hatR\longrightarrow \mathbb{K}[[\alpha_i]]$ be the $i$-th projection map defined by $\pi_i(\alpha_j)=0$ for $j\neq i$ and $\pi_i(\alpha_i)=\alpha_i$. The $i$-th ray order of $I$ is a non-negative integer $\nu=\rord_i(I)$ such that $\pi_i(I)=(\alpha_i^\nu)$.
\end{definition}
Since $\KK[[ \alpha_i]]$ is a discrete valuation ring, and all its ideals are of the form $(\alpha_i^\nu)$ for some $\nu\geq 0$, the ray order is well-defined. Let $\mathfrak{p}_i$ denote the kernel of the map $\pi_i$, the ideal generated by all $\alpha_j$ for $j\neq i$. Note that the equality $\nu =\rord_i(I)$ in particular implies that $I$ contains an element of the form $\alpha_i^\nu-q$ where $q\in \mathfrak{p}_i$.
\begin{definition}\label{raydec}
    Let $I\subset \hatR$ be an ideal of finite colength. A ray decomposition of $I$ with respect to $\alpha_i$ is an ideal $J\subset \hatR$ such that $J\subset I\cap \frakp_i$ together with an element $q\in \frakp_i$ and $\nu\in \mathbb{N}$ such that $I=J+(\alpha_i^\nu-q)\hatR$.
\end{definition}
By Definition \ref{iray}, it follows that for every ideal $I$ and $i$, a ray decomposition of $I$ with $J=I\cap \frakp_i$ and $\nu=\rord_i(I)$ exists. Moreover, for every ray decomposition of $I$ we may assume that $q\in R$, as the ideals $I$ and $J$ are $\mathfrak{m}$-primary. One can associate two families of schemes with a ray decomposition of an ideal as follows:
\begin{definition}\label{rayfamily}
    Let $I=J+(\alpha_i^\nu-q)\hatR$ be a ray decomposition of an ideal of finite colength, and let $J_R=J\cap R$. 
    The associated lower ray family is 
    \begin{equation*}
  \mathbb{K}[t] \to \frac{R[t]}{J_R[t] + (\alpha_i^{\nu} - t \alpha_i- q) R[t]}\text{,}
\end{equation*}
and the associated upper ray family is \begin{equation*}
  \mathbb{K}[t] \to \frac{R[t]}{J_R[t] + (\alpha_i^{\nu} - t \alpha_i^{\nu-1} - q) R[t]}\text{.}
\end{equation*}
When it is flat, the lower (upper) ray family is called a lower (upper) ray degeneration.
\end{definition}
\subsection{Hilbert schemes}
In this subsection we recall known results on smoothability of algebras and irreducibility of Hilbert schemes that will be used in the proofs of our results. 

\begin{theorem}[\cite{ccvv}]\label{8points}
    The scheme $\Hilb_d(\mathbb{A}^n)$ is irreducible for $d\le 7$ and has exactly two components for $d=8$ and $n\ge 4$.
\end{theorem}

\begin{theorem}[\cite{DJNT}]\label{11points}
    The scheme $\Hilb_d(\mathbb{A}^3)$ is irreducible for $d\le 11$.
\end{theorem}

\begin{theorem}[\cite{casnati_jelisiejew_notari}]\label{Gorenstein}
    All Gorenstein algebras of degree $d\le 13$ are smoothable.
\end{theorem}

\begin{proposition}[{\cite[Proposition 4.9]{ccvv}}]\label{Hf:1n1...1}
    Algebras with Hilbert functions $(1,n,1,\ldots ,1)$ are smoothable.
\end{proposition}

\begin{proposition}[{\cite[Propositions 4.11, 4.12 and remark on the top of page 774]{ccvv}}]\label{Hf:1n21}
    Algebras with Hilbert functions $(1,n,2,1)$ are smoothable.
\end{proposition}

\begin{proposition}[{\cite[Proposition 4.13]{ccvv}}]\label{Hf:1n22}
    Algebras with Hilbert functions $(1,n,2,2)$ are smoothable.
\end{proposition}

\begin{theorem}[{\cite[Proposition 4.10]{ccvv}, \cite[Problem XVI]{openproblem}, \cite{Shafarevich}]}]\label{Shafarevich}
    Algebras with Hilbert functions $(1,n,r)$ are smoothable for $r\le 2$, while they form a generically reduced elementary component of $\Hilb_{1+n+r}(\mathbb{A}^n)$ if $3\le r\le \frac{(n-1)(n-2)}{6}+2$, with the possible exception of $(1,5,4)$.
\end{theorem}

\begin{remark}\label{exception}
    The main theorem of \cite{Shafarevich} is stated without the exception $(1,5,4)$, but note that it is mentioned in the introduction of \cite{Shafarevich}, and that all preliminary lemmas needed to prove \cite[Theorem 2]{Shafarevich} assume $n\ne 5$. See also \cite[Remark 7.2]{GGGL25}.
\end{remark} 

\subsection{Variety of commuting matrices}
We recall the relation between varieties of commuting matrices and Hilbert schemes. We refer to \cite{HJ,JS} for details.

Let $\mathbb{M}_d$ denote the set of all $d\times d$ matrices and let
$$\CommMat{d}{n}=\{(A_1,\ldots ,A_n)\in \mathbb{M}_d^n:\ A_iA_j=A_jA_i\,\  \mathrm{for}\, \mathrm{all}\  i,j\}$$
be the variety of $n$-tuples of $d\times d$ commuting matrices. More precisely, we define $\CommMat{d}{n}$ as the scheme defined by quadratic equations that entry-wise describe commutativity.

Let $\cU=\CommMat{d}{n}\times \KK^d$. We say that a point $(A_1,\ldots ,A_n,v)\in \cU$ is stable if $v$ generates $\KK^d$ as a $\KK[A_1,\ldots ,A_n]$-module. In this case we also say that $v$ is a cyclic vector for the $n$-tuple $(A_1,\ldots ,A_n)$. The set of stable tuples in $\cU$ is denoted by $\cU^{\mathrm{st}}$, and it is an open subscheme of $\cU$.

The vector space $\KK^d$ becomes a module over $R=\KK[\alpha_1,\ldots ,\alpha_n]$ for the action defined by $\alpha_i\cdot u:=A_iu$. If $(A_1,\ldots ,A_n,v)\in \cU^{\mathrm{st}}$, then the map $f\mapsto f(A_1,\ldots ,A_n)\cdot v$ is an $R$-module homomorphism $R\to \KK^d$. Moreover, it is surjective, as the tuple $(A_1,\ldots ,A_n,v)$ is stable. Let $I$ be the kernel of this map. Then $\KK^d$ is isomorphic to $R/I$ as an $R$-module, while $\KK[A_1,\ldots ,A_n]$ is isomorphic to $R/I$ as an algebra.

\begin{lemma}[{\cite[Proposition 3.7]{JS}}]
    The map $(A_1,\ldots ,A_n,v)\mapsto R/I$ gives a morphism of schemes $\cU^{\mathrm{st}}\to \Hilb_d(\mathbb{A}^n)$. Moreover, $\cU^{\mathrm{st}}$ is a principal $\GL_n$-bundle over $\Hilb_d(\mathbb{A}^n)$.
\end{lemma}

\begin{corollary}[{\cite[\S 3.4]{JS}}]\label{dim_matrices-Hilbert_schemes}
    The irreducible components of $\Hilb_d(\mathbb{A}^n)$ are in bijection with irreducible components of $\cU^{\mathrm{st}}$, which are in bijection with those irreducible components of $\CommMat{d}{n}$ that contain an $n$-tuple of matrices having a cyclic vector.

    Moreover, if $\mathcal{Z}^H$ is an irreducible component of $\Hilb_d(\mathbb{A}^n)$ and $\mathcal{Z}^C$ is the corresponding irreducible component of $\CommMat{d}{n}$, then $\dim \mathcal{Z}^H=d-d^2+\dim \mathcal{Z}^C$.
\end{corollary}

We now describe the relation between the support of the algebra in $\Hilb_d(\mathbb{A}^n)$ and eigenvalues of the corresponding commuting matrices, see \cite[\S 3.3 - 3.4]{JS}. Let $R/I\in \Hilb_d(\mathbb{A}^n)$ and let $(A_1,\ldots ,A_n)$ be the corresponding commuting matrices. If the maximal ideal $(\alpha_1-\lambda_1,\ldots ,\alpha_n-\lambda_n)$ is in the support of $R/I$, then $\lambda_i$ is an eigenvalue of $A_i$ for each $i$. Moreover, the support of $R/I$ contains only one point if and only if every matrix $A_i$ has only one eigenvalue. In particular, algebras supported at zero correspond exactly to nilpotent commuting matrices. Consequently, the elementary components of $\Hilb_d(\mathbb{A}^n)$, that is those parametrizing algebras supported at one point, correspond to components of $\CommMat{d}{n}$ containing only $n$-tuples of matrices with a single eigenvalue. On the other hand, the smoothable component of $\Hilb_d(\mathbb{A}^n)$ corresponds to the principal component of $\CommMat{d}{n}$, that is the closure of the locus of $n$-tuples of simultaneously diagonalizable matrices. We note that the dimension of the principal component of $\CommMat{d}{n}$ is $d^2+(n-1)d$.\\

\begin{remark}\label{3.20}
There is a close relation between Hilbert function of a local algebra supported at zero and the dimension of the common cokernel (i.e.\ the common kernel of transposes) of corresponding commuting matrices. Similarly, the socle dimension can be realized as the dimension of the common kernel of the matrices. We will frequently use this correspondence, so we recall it now, and refer to \cite[Lemma 3.19 and Equations (3.20)]{JS}.

Let $(\cA,\mathfrak{n})$ be a local algebra supported at zero, and $(A_1,\ldots ,A_n)\in \CommMat{d}{n}$ the corresponding $n$-tuple of nilpotent commuting matrices. If $k$ is any positive integer, and $\{i_1,\ldots ,i_k\}$ a subset of indices, then
$$\dim_{\KK}\cA/\mathfrak{n}^k=\dim_{\KK}\bigcap_{i_1,\ldots ,i_k}\ker (A_{i_1}\cdots A_{i_k})^T$$ 
and $$\tau (\cA)=\dim_{\KK}\bigcap_{i=1}^n\ker A_i.$$
In particular, the common cokernel of the matrices $A_1,\ldots ,A_n$ is 1-dimensional, and the algebra $\cA$ is Gorenstein if and only if the common kernel of matrices is 1-dimensional.
\end{remark}

In the proofs of the main results we will frequently compute the tangent spaces to $\CommMat{d}{n}$. We recall the following two facts.

\begin{lemma}[{\cite[Lemma 3.1]{JS}}]\label{tangent_space_commuting_marices}
    If $(A_1,\ldots ,A_n)\in \CommMat{d}{n}$, then
    $$T_{(A_1,\ldots ,A_n)}\CommMat{d}{n}=\{(Z_1,\ldots ,Z_n)\in \mathbb{M}_d^n:\ \ A_iZ_j+Z_iA_j=A_jZ_i+Z_jA_i\,\  \mathrm{for}\, \mathrm{all}\,\  i,j\}.$$
\end{lemma}

\begin{lemma}[{\cite[Lemma 3.9]{JS}}]\label{dim_tangent_spaces}
    Let $R/I\in \Hilb_d(\mathbb{A}^n)$ be arbitrary and let $(A_1,\ldots ,A_n)\in \CommMat{d}{n}$ be the corresponding tuple of commuting matrices. Then
    $$\dim T_{R/I}\Hilb_d(\mathbb{A}^n)=d-d^2+\dim T_{(A_1,\ldots ,A_n)}\CommMat{d}{n}.$$
\end{lemma}

We get an immediate corollary of Lemmas \ref{dim_matrices-Hilbert_schemes} and \ref{dim_tangent_spaces}.

\begin{corollary}
    A point $R/I$ is smooth in $\Hilb_d(\mathbb{A}^n)$ if and only if the corresponding point $(A_1,\ldots ,A_n)$ is smooth in $\CommMat{d}{n}$.
\end{corollary}

\section{Description of graded Hilbert schemes}\label{DescriptionOfComponents}
Fix a one-dimensional torus $\Gmult$ and its action on
    $\mathbb{A}^n$ by $t\cdot (x_1, \ldots ,x_n) = (tx_1, \ldots ,
    tx_n)$. The
    points of $\Hilb_d^{\Gmult}(\mathbb{A}^n,0)$ are just subschemes $Z
    \subset \mathbb{A}^n$ whose ideals are homogeneous with respect to the
    standard grading. Let $\HilbFunc{H}{n}$ be the locus of local algebras supported at the origin and $\Hilb_H^{\Gmult}(\mathbb{A}^n,0)$ be the locus of local graded algebras with Hilbert function $H$, respectively. There is a natural map
$$\pi_H:\HilbFunc{H}{n}\longrightarrow  \Hilb_H^{\Gmult}(\mathbb{A}^n,0)$$
mapping an algebra to its associated graded algebra or an ideal $I$ to its initial ideal with respect to the weight vector $(-1,\ldots,-1)$.
In this section, we describe the irreducible components of $\HilbFunc{H}{n}$ for $H=(1,n,3,2)$ and $H=(1,n,2,2,1)$, using the description of the components of $\Hilb_H^{\Gmult}(\mathbb{A}^n,0)$ and the fibres of the map $\pi_H$.

\begin{definition}
    We say that a polynomial $f \in \mathbb{K}[x_1,\ldots,x_n]$ depends on $k$ variables if there exist linear forms $y_1,\ldots,y_k \in \mathbb{K}[x_1,\ldots,x_n]$ such that $f \in \mathbb{K}[y_1,\ldots,y_k]$. We say that $f$ depends essentially on $k$ variables if it depends on $k$ variables but does not depend on $k-1$ variables.
\end{definition}

\subsection{Hilbert function $(1,n,3,2)$}
In order to analyze the scheme $\HilbFuncGr{H}{n}$ for Hilbert function $H = (1,n,3,2)$, we need to classify spaces of polynomials with special properties and use some classical algebraic geometry (secant varieties). We do so in Lemma~\ref{better_pencil_form}, Proposition~\ref{buchsbaum_eisenbud}, and Lemma~\ref{annoying_cubics}.

\begin{lemma}
    \label{better_pencil_form}
    Suppose $v_3 : \PP(\KK^2) \hookrightarrow \PP(\Sym^3 \KK^2)$ is the Veronese embedding and $L \subseteq   \PP^3=\PP(\Sym^3\KK^2) $ is a line not intersecting $v_3(\PP^1)$. Then $L$ is spanned by two points of the form
    \begin{equation*}
        [x^3 + a y^3] \text{ and } [x^3 + b(x+y)^3]
    \end{equation*}
    for some coordinates $x,y$ on $\PP^1$ and some $a,b \in \mathbb{K}$.
\end{lemma}
\begin{proof}
   A general plane through $L$ cuts $v_3(\PP^1)$ in $3$ points. We can choose the points to be $[x^3], [y^3]$ and $[(x+y)^3]$. The lines $\langle x^3, y^3 \rangle $ and $L$ intersect (two lines on $\PP^2$), and lines $\langle x^3, (x+y)^3\rangle$ and $L$ also intersect. The lemma follows.
\end{proof}

\begin{proposition}
    \label{buchsbaum_eisenbud}
    Consider the Veronese embedding $v_3 : \PP^2 \hookrightarrow \PP^9$. Let $U \subseteq \PP (\Sym^3 \KK^3)$ be the set of cubics $F$ depending essentially on $3$ variables whose annihilator $\Ann(F)$ has at least one cubic minimal generator. Then
    \begin{enumerate}
        \item For every cubic $[F]$ in $U$ the ideal $\Ann(F)$ has exactly two minimal cubic generators.
        \item The closure of $U$ is the secant variety $\sigma_3(v_3(\PP^2))$. In particular, $\overline{U}$ is irreducible of dimension $8$.
    \end{enumerate}
\end{proposition}
\begin{proof}
  Let us begin with Point~1. The Betti table of minimal resolution of $R/\Ann(F)$ is symmetric. It has four columns (since we are over the ring $\KK[\alpha_1,\alpha_2, \alpha_3]$) and four rows (by the structure of the Buchsbaum-Eisenbud resolution \cite[Theorem~B.2]{iakanev}), so it is of the form
  \begin{equation}
      \label{betti_table}
      \begin{matrix}
          b_{00} & - & - & - \\
          - & b_{12} & b_{23} & - \\
          - & b_{13} & b_{24} & - \\
          - & - & - & b_{36} 
      \end{matrix}
  \end{equation}
  From \cite[Section~1.4]{boij_sorderberg} it follows that the possible pure Betti tables of shape \eqref{betti_table} for Cohen-Macaulay modules of codimension $3$ are
  \begin{equation*}
      \begin{matrix}
          1 & - & - & - \\
          - & 3 & - & - \\
          - & - & 3 & - \\
          - & - & - & 1 
      \end{matrix}\hspace{1cm}
      \begin{matrix}
          2 & - & - & - \\
          - & 9 & 8 & - \\
          - & - & - & - \\
          - & - & - & 1 
      \end{matrix}\hspace{1cm}
      \begin{matrix}
         1 & - & - & - \\
         - & - & - & - \\
         - & 8 & 9 & - \\
         - & - & - & 2 
      \end{matrix}
  \end{equation*} 
  The Betti table of $R/\Ann(F)$ is a positive rational combination of those tables (see \cite[Section~1.5]{boij_sorderberg}). We know that $\Ann(F)$ has three minimal quadratic generators, so we get that
  \begin{equation*}
  \begin{pmatrix}
          1 & - & - & - \\
          - & 3 & b_{23} & - \\
          - & b_{13} & 3 & - \\
          - & - & - & 1 
  \end{pmatrix}
  = a\cdot 
      \begin{pmatrix}
          1 & - & - & - \\
          - & 3 & - & - \\
          - & - & 3 & - \\
          - & - & - & 1 
      \end{pmatrix}
      + b\cdot 
      \begin{pmatrix}
          2 & - & - & - \\
          - & 9 & 8 & - \\
          - & - & - & - \\
          - & - & - & 1 
      \end{pmatrix}
      + c \cdot
      \begin{pmatrix}
         1 & - & - & - \\
         - & - & - & - \\
         - & 8 & 9 & - \\
         - & - & - & 2 
      \end{pmatrix}
  \end{equation*}
  for some $a,b,c \geq 0$. From this we obtain $b = c$, and then $b \leq \frac{1}{3}$, and this implies $b_{13} = b_{23} \in \{0,1,2\}$. By \cite[Theorem~B.2]{iakanev}, the ideal $\Ann(F)$ has an odd number of generators, therefore it must have exactly two minimal cubic generators.

  As for Point~2., by \cite[Theorem~1.7]{apolarity_direct_sums}, we know that $F$ must be a limit of direct sums, which in this case means that $F$ is in the third secant variety.
\end{proof}

\begin{lemma}
  \label{annoying_cubics}
  Let $n \geq 3$. Let $W\subseteq \Sym^3 \KK^n$ be a pencil of cubics such that for every $f \in W$ there exist linear forms $x,y$ with $f \in \Sym^3 \langle x, y \rangle$. 
  Then either all the cubics in $W$ depend on the same set of $2$ variables, or we can write $W = \langle x^2 y , x^2 z\rangle$ for some linearly independent forms $x,y,z$.
\end{lemma}
\begin{proof}
    We consider a basis of $W$ given by two cubics $f_1$ and $f_2$, both of them essentially in two variables. Let $V_i\subseteq \KK^n$ be the vector space of variables, on which $f_i$ depends, for $i = 1,2$. If $V_1 \cap V_2 = 0$, then the cubic $f_1 + f_2$ depends essentially on $4$ variables, a contradiction. 
    If $V_1 \cap V_2 = 2$, then all the cubics in $W$ depend on the same set of $2$ variables.
    
    Suppose $\dim (V_1 \cap V_2) = 1$. We may assume $V_1 = \langle x,y\rangle$ and $V_2 =\langle x, z\rangle$. The cubic $f_1$ can be either of rank $2$ or $3$. If it is of rank two, then either $f_1 = x^3 + y^3$ or $f_1 = \lambda y^3 + (x+y)^3$, depending on whether $x$ appears in the decomposition of $f_1$. If it is of rank three, then $f_1 = u^2 v$ for some linear forms $u, v$. We get that $ f_1 \in \{x^2 y, y^2 x, y^2(x+y)\}$, depending on whether $x$ is one of the forms $u$, $v$ (first two cases) or not (third case). We obtain that
    \begin{equation*}
      f_1 \in \{ x^3 + y^3, \lambda y^3 + (x+y)^3, x^2y, y^2 x, y^2 (x+y) \}
    \end{equation*}
    for some $\lambda \in \mathbb{K}$.
    Similarly
    \begin{equation*}
      f_2 \in \{ x^3 + z^3, \mu z^3 + (x+z)^3, x^2z, z^2 x, z^2 (x+z) \}
    \end{equation*}
    for some $\mu \in \mathbb{K}$. This is a finite number of cases to check, plus an infinite number of cases parametrized by $\mu \in \KK$ when $f_2 = \mu z^3 + (x+z)^3$ (or $f_1 = \lambda y^3 + (x+y)^3$, but this is symmetric). Here we check the cases coming as a family:
    \begin{enumerate}[label=\arabic*.]
    \item $f_1 = x^3 + y^3, f_2 = \mu z^3 + (x+ z)^3$

      \begin{equation*}
        (a \alpha + b \beta + c \gamma) \lrcorner (f_1 + f_2) 
        = 3x^2(2a + c) + 3y^2 b + 3z^2(a + \mu c + c) + 6xz(a+c)\text{,}
      \end{equation*} 
    \item $f_1 = \lambda y^3 + (x+y)^3, f_2 = \mu z^3 + (x + z)^3$
      \begin{multline*}
         (a \alpha + b\beta + c \gamma) \lrcorner(f_1 + f_2)  \\= 3x^2(2a + b + c) + 3y^2(a + \lambda b + b) + 3z^2(a + c\mu + c) + 6xy(a+b) + 6xz(a+c)\text{,}
      \end{multline*}
      \item $f_1 = x^2 y, f_2 = \mu z^3 + (x+ z)^3$
      \begin{equation*}
          (a\alpha + b\beta + c \gamma) \lrcorner (f_1 + f_2) = x^2(3a + b + 3c) + z^2(3a + 3c\mu + 3c) + xy \cdot 2a + xz (6a + 6c),
      \end{equation*}
      \item $f_1 = y^2 x, f_2 = \mu z^3 + (x+ z)^3$
        \begin{equation*}
            (a\alpha + b\beta + c\gamma) \lrcorner (f_1 + f_2) = x^2(3a + 3c) + y^2 a + z^2 (3a + 3\mu c + 3c) + xy \cdot 2b + xz(6a + 6c),
        \end{equation*}
      \item $f_1 = y^2(x+y), f_2 =\mu z^3 + (x + z)^3$
      \begin{equation*}
          (a\alpha + b\beta + c\gamma) \lrcorner (f_1 + f_2) = x^2 (3a + 3c)  + y^2 (a + 3b) + z^2 (3a + 3\mu c + 3c) + xy\cdot 2b + (6a + 6c)xz.
      \end{equation*}
    \end{enumerate}
      In each of these cases, from $(a\alpha + b\beta +c\gamma)  \lrcorner (f_1 + f_2) = 0$ it follows that $a = b = c = 0$, in other words, $f_1 + f_2$ depends essentially on $3$ variables. After checking the remaining finitely many cases, we can see that the only pair, in which $f_1 + f_2$ depends on two variables, is $f_1 = x^2 y, f_2 = x^2 z$.
\end{proof}

The following proposition is along \cite[Proposition~5.12]{JK22}, though we further describe the irreducible components.

\begin{proposition}\label{HF1m32graded}
  Let $H = (1,n,3,2)$ and $n\geq 4$. Then $\Hilb_H^{\mathbb{G}_m}(\AA^n, 0)$ consists of the following irreducible loci:
  \begin{enumerate}[label=(\roman*)]
    \item ideals orthogonal to two rank one cubics $x^3, y^3$ and one quadric,
    \item ideals whose inverse system contains a cubic form depending essentially on $3$ variables,
    \item ideals orthogonal to a pencil of cubics, each depending on the same set of $2$ variables, and
      exactly one of those cubics is of rank one,
    \item ideals orthogonal to a pencil of cubics, each depending essentially on the same set of $2$ variables,
    \item ideals orthogonal to the pencil $\langle x^2 y, x^2 z\rangle$ in some coordinates.
  \end{enumerate}
  The dimensions of the loci are given in the following table:
    \begin{center}
    \begin{tabular}{|c|c|}
      \hline
      (i) &  $2(n-1) + \binom{n+1}{2}-3$  \\
      \hline 
      (ii) & $3n -1$ \\
      \hline
      (iii)  & $2n -1$ \\
      \hline
      (iv) & $2n$  \\
      \hline
      (v) & $3n -5 $ \\
      \hline

    \end{tabular}
  \end{center}
  Items (iii), (iv), (v) deform to item (ii). The closures of loci (i) and (ii) are the two irreducible components of
  $\Hilb_H^{\mathbb{G}_m}(\AA^n,0)$.
\end{proposition}
\begin{proof}
We know that in this case the inverse system of an ideal in degree three $(I^{\perp})_3$ is generated by two cubic polynomials. The cubics in the pencil must depend on the same set of 3 variables, as otherwise their derivatives would span at least a 4-dimensional subspace of $(I^{\perp})_2$, contradicting $H(2)=3$. We consider the following cases: 

\textbf{Case 1.} Assume the inverse systems contains a cubic $F$ essentially in three variables, i.e.\ we analyze Case (ii). We have to pick one more cubic generator of the inverse system or, dually, choose a codimension one subspace in the space of minimal cubic generators of $\Ann(F)$. This means that
    $\Ann(F)$ has at least one minimal cubic generator, but by Proposition \ref{buchsbaum_eisenbud}, Item 1., it must have exactly
    two minimal cubic generators. Let 
    \begin{equation*}
        Z = \{[F] \in  \PP(\Sym^3 \KK^n) | \Ann(F) \text{ has two minimal cubic generators and} \dim (R/\Ann(F))_1 = 3\}.
    \end{equation*}
    The set $Z$ can be described as a fibration over $\Gr(3, n)$ with every fiber isomorphic to an non-empty
    open subset of $\sigma_3(v_3(\PP^2))$ (by Proposition \ref{buchsbaum_eisenbud}, Item 2). Thus, $Z$ is irreducible of dimension $3(n-3) + 8$. Let $Y$ be the incidence
    variety
    \begin{equation*}
      Y = \{(I, [F]) \in \Hilb_{H}^{\mathbb{G}_m}(\AA^n, 0) \times Z\ :\ I \subseteq \Ann(F)\}.
    \end{equation*}
    We denote by $\pi_1$ the projection $Y \to \Hilb_{H}^{\mathbb{G}_m}(\AA^n, 0)$ and by $\pi_2$ the  projection $Y \to Z$. Every fiber
    of $\pi_2$ is isomorphic to $\PP^1$, which follows from the fact that, if we fix $F$, a cubic essentially in three variables, we have to choose a codimension one subspace in the $2$-dimensional space of minimal cubic generators of $\Ann(F)$.  Hence, we get that $Y$ is irreducible and of dimension $\dim Z  + 1 = 3n$.
    Every fiber of $\pi_1$ is isomorphic to an open subset of $\PP^1$, hence the
    locus (ii) is irreducible and of dimension $3n -1$. 
    
    A general element of this locus is of the form $\Ann(x^3 - a y^3, x^3 - b z^3,l_1,\ldots,l_{n-3})$ for some $a,b \in \KK$ and some linear forms $l_i$. To see this, first assume $n = 3$. Then, by Proposition~\ref{buchsbaum_eisenbud}, Item~2, a general cubic  $F$ in the considered inverse system is of the form $x^3 + y^3 + z^3$, meaning that $\Ann(F) = (\alpha \beta, \beta \gamma, \alpha \gamma, \alpha^3 - \beta^3, \alpha^3 - \gamma^3)$. Choosing a one dimensional subspace of cubic generators amounts to picking an ideal of the form $(\alpha \beta, \beta \gamma, \alpha \gamma, a \alpha^3 + b \beta^3 + (-a - b) \gamma^3)$, whose inverse system is $\Ann(bx^3 - a y^3, (-a-b) x^3 - a z^3)$. What is left is rescaling $x,y,z$ and adding some linear forms in the inverse system if we increase $n$.
    \\

   \textbf{Case 2.} Now assume that there is no cubic depending essentially on $3$ variables in the pencil. We have three cases:
    \begin{enumerate}[label = \arabic*.]
      \item There are two cubics of rank one in the pencil. In this case, we get Locus (i). We have $2.\dim \Gr(1,n)=2(n-1)$ parameters for the choice of two linear forms $x,y$, and $\binom{n+1}{2}-3$ parameters for the choice of a
	quadric modulo $x^2,y^2$.
      \item There is exactly one cubic of rank one in the pencil.

	By Lemma \ref{annoying_cubics}, the cubics depend on the same set of $2$ variables. Note that their derivatives span a 3-dimensional space, so they already span $(I^{\perp})_2$. We perform the dimension count
	in two steps. First we choose a two-dimensional space in $\KK^n$, and then we have to count the number of
	inverse systems in two variables, containing exactly one power. To this end, consider the abstract join of
	lines:
     \begin{equation*}
         \overline{ \{ (L, x,y) \in \Gr(2, 4)\times v_3(\PP^1) \times \PP^3\ :\ x \in L,\ y \in L ,\ x \neq y \}  }.
     \end{equation*}
     Since every pair of points lies on exactly one line, this join is irreducible of dimension $4$. When we project onto $\Gr(2, 4)$, we get our Locus (iii). Since a general fiber of this projection is one-dimensional, we get the irreducibility and the dimension count.
      \item There are no cubics of rank one in the pencil.

	First assume that every cubic in the pencil depends essentially on the same space of variables $\langle
	x,y\rangle$. Here the parameter count is the following: we choose a subspace in $\Gr(2, n)$ and choose two general
	cubics in the subspace. This gives $2(n-2) + 4 = 2n$ choices.

	By Lemma \ref{annoying_cubics}, if the cubics depend on different sets of variables, we get the case
	$\langle x^2 y, x^2z\rangle$. For parameter count, we choose a linear form $x$ for which we have $\dim \PP^{n-1} = n-1$ parameters,  and a subspace $\langle y, z \rangle$ which counts to $\dim \Gr(2, n) = 2n -4$ parameters.
    \end{enumerate}
    In the following table, we summarize all the degenerations:
    \begin{center}
    \begin{tabular}{|c|c|c|}
      \hline
      $(ii) \to (iv)$ & $\langle x^3 - ay^3, x^3 - bz^3\rangle\to\langle x^3 - a y^3, x^3 - b(x+y)^3\rangle $ & $z\to
      x+y$ \\
      \hline 
      $(ii) \to (v)$ & $\langle(x + \lambda y )^3 - x^3, (x + \lambda z)^3 - x^3\rangle \to \langle x^2 y, x^2 z\rangle $ & $\lambda \to 0 $ \\
      \hline
      $(iv) \to (iii) $ & degenerate one of the cubics to a rank one cubic & \\
      \hline
    \end{tabular}
  \end{center}
  Notice that by Lemma~\ref{better_pencil_form}, a general pencil in Locus (iv) can be written as $\langle x^3 - a y^3, x^3 - b (x+y)^3 \rangle$ for some $a,b \in \KK$.

  By dimension comparison, we see that the closure of Locus (i) is an irreducible component of $\HilbFuncGr{H}{n}.$ The closure of Locus (ii) is an irreducible component of $\HilbFuncGr{H}{n}$, since an algebra of socle type $1^{n-3}\cdot 3^2$ cannot deform to an algebra of socle type $2\cdot 3^2$.
  \end{proof}

\subsection{Hilbert function $(1,n,2,2,1)$}

We start by recalling some geometry of binary forms. Let $v_4 : \PP(\mathbb{K}[x,y]_1) \to \PP(\mathbb{K}[x,y]_4)$ be the Veronese embedding. Let 
\begin{align*}
    \tau_2(v_4(\PP^1))  &= \overline{\{[F]\in \PP(\mathbb{K}[x,y]_4) \ :\ F = u^3 v \text{ for some linear forms } u, v \text{ on } \PP^1\}}, \\
    \sigma_2(v_4(\PP^1)) &= \overline{\{[F] \in \PP(\mathbb{K}[x,y]_4) \ :\ F = u^4 + v^4 \text{ for some linear forms } u, v \text{ on } \PP^1\}}
\end{align*}
be the tangential and the secant variety, respectively, so that
\begin{equation*}
    \tau_2(v_4(\PP^1)) \subset \sigma_2(v_4(\PP^1)).
\end{equation*}
 It is a classical fact that $\dim \tau_2(v_4(\PP^1)) = 2$ and $\dim \sigma_2(v_4(\PP^1)) = 3$.

\begin{proposition}\label{14221}
Let $H=(1,n,2,2,1)$. The Hilbert scheme $\HilbFuncGr{H}{n}$ consists of the following irreducible loci:
\begin{enumerate}[label=(\roman*)]
\item\label{l4l2l'} ideals orthogonal to $x^4$ and $x^2 y$, such that $x$ and $y$ are linearly independent forms. This is an irreducible subset of dimension $\dim \Gr(2,n) + \dim \PP^1 = 2(n-2) + 1$.
\item\label{2'} ideals orthogonal to a special quartic of the form $x^3 y$, where $x$ and $y$ are linearly independent forms. This is an irreducible subset of dimension $\dim \Gr(2,n) + \dim \tau_2(v_4(\PP^1)) = 2(n-2) +2$.
\item\label{l4l'3} ideals orthogonal to $x^4$ and a cubic of the form $y^3$, where $x$ and $y$ are linearly independent forms. This is an irreducible subset of dimension $2(n -2) + 2$.
\item\label{l4+l'4} ideals orthogonal to a quartic of the form $x^4 + y^4$, for two  linearly independent forms $x,y$. This is an irreducible subset of dimension $\dim \Gr(2,n) + \dim\sigma_2(v_4(\PP^1)) = 2(n-2) + 3$.
\end{enumerate}
Moreover, we have the following diagram of degenerations:
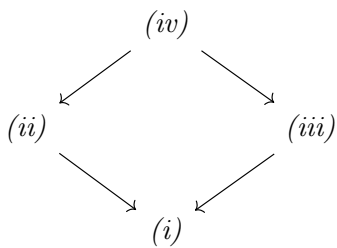
\[\begin{tikzcd}
   & \text{\textit{(iv)}}\ar{ld}\ar{rd} & \\
   \text{\textit{(ii)}}\ar{rd} &  & \text{\textit{(iii)}}\ar{ld} \\
   & \text{\textit{(i)}} & 
\end{tikzcd}\]Hence, the scheme $\HilbFuncGr{H}{n}$ is irreducible.
\end{proposition}
\begin{proof}
    The degree 4 part of the inverse system of an ideal is spanned by a quartic form in two variables, since the space of its partial derivatives is at most two-dimensional. Therefore we first pick a two-dimensional subspace $W\subseteq \mathbb{K}^n$.
    We have either of the two following cases:
    \begin{enumerate}[label=(\alph*)]
    \item the inverse system consists of a degree $4$ form in $\Sym^4 W$ depending essentially on two variables, and some forms of degree $1$.
    \item the inverse system consists of a degree $4$ form $x^4 \in\Sym^4 W$, a form $g \in \Sym^3 W$ such that $x^3$ and $g$ are linearly independent, and some forms of degree $1$.
    \end{enumerate}
    In Case (a), we choose a form in $\Sym^4 W$ whose apolar has Hilbert function $(1,2,2,2,1)$, or equivalently, a form in the second secant variety of the fourth Veronese of $\PP^1$. Depending on whether we choose a form of rank two or rank four, we get the subsets \ref{l4+l'4} or \ref{2'}. The subsets are irreducible, because they are bundles over $\Gr(2, n)$ with each fiber isomorphic $\sigma_2(v_4(\PP^1))\setminus \tau_2(v_4(\PP^1))$ or $\tau_2(v_4(\PP^1))$, respectively. Since the dimensions of the secant and the tangential variety are $3$ and $2$, respectively, we get the description of the irreducible subsets \ref{l4+l'4} and \ref{2'}.

    In Case (b), we choose a form $x$ in $\Sym^4 W$ whose apolar has Hilbert function $(1,1,1,1,1)$, so we must choose another form in the third graded piece as well. As $H(2)=2$,  we can pick either $y^3$ or $x^2 y$. In the former case, we have to choose forms $x$ and $y$, so it is a two dimensional set of choices. In the latter case, the choice of the form $y$ does not matter, because
    \begin{equation*}
        a x^3 + b x^2 y = x^2(a x + by)\text{,}
    \end{equation*}
    so when we choose one linear form $y$, it is the same as if we chose any other form $ax + by$. Thus we get irreducible subsets \ref{l4l'3} and \ref{l4l2l'}, as desired.

    As for degenerations, we know that \ref{l4+l'4} degenerates to \ref{2'}. Moreover,
    \begin{equation*}
      \Ann(x^4, (a x + y)^3) = \Ann(x^4, a^3x^3 + 3a^2 x^2 y + 3 axy^2 + y^3) = \Ann(x^4, x^2 y + \frac{1}{a} xy^2 + \frac{1}{3a^2}  y^3)\text{,}
    \end{equation*}
    which goes to $\Ann(x^4, x^2 y)$ when $a \to \infty$. Thus \ref{l4l'3} degenerates to \ref{l4l2l'}. Moreover, the ideal \ref{l4+l'4} degenerates to \ref{l4l'3} with $\Ann(x^4 + ay^4)\to \Ann(x^4, y^3)$ when $a\to 0$. Similarly, \ref{2'} goes to \ref{l4l2l'} with $\Ann(x^3 (x + ay))\to \Ann(x^4, x^2 y)$ when $a \to 0$.
\end{proof}
\begin{proposition}\label{fiberdim} Let $n = 4$. Fibers over each of the given irreducible subsets in Proposition~\ref{14221} are irreducible. The fibers over \ref{l4l2l'}, \ref{l4l'3} are $19$-dimensional, yet the fibers over \ref{2'} and \ref{l4+l'4} are $14$-dimensional. In particular,  $\HilbFunc{H}{4}$ has exactly two components $P$ and $Q$, with $P$ defined as the preimage of the loci \ref{l4l2l'} and \ref{l4l'3}, and $Q$ as the closure of the preimage of \ref{2'} and \ref{l4+l'4}, respectively. 
\end{proposition}
\begin{proof}
Let $R = \KK[\alpha, \beta,\gamma, \delta]$ and the dual ring is $S = \KK[x,y,z,w]$ as in Subsection~\ref{Macaulay's Inverse System}.


Suppose we are in Case \ref{l4l2l'}. We show the fibers are irreducible of the same dimension. First, let $J$ be in the fiber over \begin{equation*}
    I=\Ann(x^4,x^2y,z,w) = (\alpha^5, \alpha^3 \beta, \beta^2, \gamma \alpha, \gamma \beta, \gamma^2, \delta \alpha, \delta \beta, \delta^2, \delta \gamma)\text{,}
\end{equation*}
which must be generated by $\mathfrak{m}^5$ and the polynomials
\begin{align*}
    &\hspace{5cm} \theta_0 = \beta^2 + a_0 \alpha^3 + b_0 \alpha^2 \beta + c_0 \alpha^4 \\
    & \theta_1=\gamma \alpha + a_1 \alpha^3+b_1\alpha^2 \beta +c_1\alpha^4 &\quad \theta_2=\delta \alpha + a_2 \alpha^3 + b_2 \alpha^2 \beta + c_2 \alpha^4\\
    &\theta_3=\gamma\beta + a_3 \alpha^3+b_3\alpha^2 \beta +c_3\alpha^4&\quad \theta_4=\delta \beta + a_4 \alpha^3+b_4\alpha^2 \beta +c_4\alpha^4\\
    &\theta_5=\gamma^2 + a_5 \alpha^3+b_5\alpha^2 \beta +c_5\alpha^4 &\quad \theta_6=\delta^2 +a_6 \alpha^3+b_6\alpha^2 \beta +c_6\alpha^4 \\
    &\hspace{5cm} \theta_7=\gamma\delta + a_7 \alpha^3+b_7\alpha^2 \beta +c_7\alpha^4 \\
    &\hspace{5cm}  \zeta =\alpha^3 \beta
\end{align*}
The coefficients $a_i,b_i,c_i \in \KK$ are not free, as we require $I$ to be the initial ideal of $J$ with respect to the weight vector $\tiny{(-1,-1,-1,-1)}$. In fact, since the monomials $\alpha^j\beta,\alpha^j\gamma,\alpha^j \delta$ are in $J$ for $j\geq 3$, and $\mathfrak{m}^{5}\subset J$, it is necessary to only look at the linear syzygies among the quadric generators of $I$, involving $\alpha$, to get required conditions on the coefficients for having $I=\inn_{_{(-1,-1,-1,-1)}}(J)$. In fact, we must have $\beta \theta_1 - \alpha \theta_3, \gamma \theta_1 - \alpha \theta_5, \delta \theta_1 - \alpha \theta_7, \beta \theta_2 - \alpha \theta_4, \delta \theta_2 - \alpha \theta_6, \gamma \theta_2 - \alpha \theta_7 \in J$, which implies that $a_3 = a_4 = a_5 = a_6 = a_7 = 0$. Moreover, by Buchberger Algorithm these are all the conditions. Hence, the fibers over \ref{l4l2l'} are irreducible of dimension $3\cdot 8-5=19$. 

Now let 
\begin{equation*}
    I = \Ann(x^4, y^3, z,w) = (\alpha^5, \beta^4, \alpha \beta, \gamma \alpha, \gamma \beta, \gamma^2, \delta \alpha, \delta \beta, \delta^2, \gamma \delta)\text{.}
\end{equation*}
An ideal $J$ in the fiber over $I$ is generated by $\mathfrak{m}^5$ and 
\begin{align*}
    &\hspace{5cm} \theta_0 = \alpha \beta + a_0\alpha^3 + b_0 \beta^3 + c_0 \alpha^4 \\
    & \theta_1=\gamma \alpha +a_1\alpha^3 + b_1 \beta^3 + c_1 \alpha^4 &\quad \theta_2=\delta \alpha +a_2\alpha^3 + b_2 \beta^3 + c_2 \alpha^4 \\
    &\theta_3=\gamma\beta +a_3\alpha^3 + b_3 \beta^3 + c_3 \alpha^4  &\quad \theta_4=\delta \beta +a_4\alpha^3 + b_4 \beta^3 + c_4 \alpha^4 \\
    &\theta_5=\gamma^2 +a_5\alpha^3 + b_5 \beta^3 + c_5 \alpha^4 &\quad \theta_6=\delta^2 +a_6\alpha^3 + b_6 \beta^3 + c_6 \alpha^4   \\
    &\hspace{5cm} \theta_7=\gamma\delta +a_7\alpha^3 + b_7 \beta^3 + c_7 \alpha^4  \\
    &\hspace{5cm}  \zeta = \beta^4
\end{align*}
With the same argument, one needs to look at the linear syzygies among the eight quadric generators $\alpha \beta, \gamma \alpha, \gamma \beta, \gamma^2, \delta \alpha, \delta \beta, \delta^2, \gamma \delta$ of the ideal $I$, but only those involving $\alpha$. Those syzygies are $\gamma \theta_0 - \alpha \theta_3, \delta \theta_0 - \alpha \theta_4, \beta \theta_1 - \alpha \theta_3, \gamma \theta_1 - \alpha \theta_5, \delta \theta_1 - \alpha \theta_7, \beta \theta_2 - \alpha \theta_4, \delta \theta_2- \alpha \theta_6, \gamma \theta_2 - \alpha \theta_7$ implying $a_3 = a_4 = a_5 = a_6 = a_7 = 0$, and hence the fibers over the ideals in \ref{l4l'3} are irreducible of dimension $3\cdot 8-5=19$.

Now, let $J$ be in the fiber over 
\begin{equation*}
  I=\Ann(x^3 y, z,w) = (\alpha^4, \beta^2, \gamma \alpha, \gamma \beta, \gamma^2, \delta \alpha, \delta \beta, \delta^2, \gamma \delta)
\end{equation*}
in the component \ref{2'}, which must be generated by $\mathfrak{m}^5$ and the generators
\begin{align*}
    &\hspace{5cm} \theta_0 = \beta^2 +  a_0 \alpha^3 + b_0 \alpha^2 \beta + c_0 \alpha^3 \beta\\
    & \theta_1=\gamma \alpha + a_1 \alpha^3 + b_1 \alpha^2 \beta + c_1 \alpha^3 \beta &\quad \theta_2=\delta \alpha +a_2 \alpha^3 + b_2 \alpha^2 \beta + c_2 \alpha^3 \beta \\
    &\theta_3=\gamma\beta + a_3 \alpha^3 + b_3 \alpha^2 \beta + c_3 \alpha^3 \beta &\quad \theta_4=\delta \beta + a_4 \alpha^3 + b_4 \alpha^2 \beta + c_4 \alpha^3 \beta \\
    &\theta_5=\gamma^2 +a_5 \alpha^3 + b_5 \alpha^2 \beta + c_5 \alpha^3 \beta &\quad \theta_6=\delta^2 +a_6 \alpha^3 + b_6 \alpha^2 \beta + c_6 \alpha^3 \beta   \\
    &\hspace{5cm} \theta_7=\gamma\delta +a_7 \alpha^3 + b_7 \alpha^2 \beta + c_7 \alpha^3 \beta \\
    &\hspace{5cm}  \zeta = \alpha^4
\end{align*}
In this case, the linear syzygies involving $\alpha$ or $\beta$ are: $ 
\gamma \theta_0 - \beta \theta_3,
\delta \theta_0 - \beta \theta_4,
\beta \theta_1 - \alpha \theta_3,
\gamma \theta_1 - \alpha \theta_5,
\delta \theta_1 - \alpha \theta_7,
\beta \theta_2 - \alpha \theta_4,
\delta \theta_2 - \alpha \theta_6,
\gamma \theta_2 - \alpha \theta_7,
\gamma \theta_3 - \beta \theta_5,
\delta \theta_3 - \beta \theta_7,
\delta \theta_4 - \beta \theta_6,
\gamma \theta_4 - \beta \theta_7.
$
Therefore, the conditions $a_3, a_4, a_5, a_6, a_7, b_5, b_6, b_7, a_1 - b_3, a_2 - b_4 = 0$ are necessary and sufficient to have $I=\inn_{_{(-1,\ldots,-1)}}(J)$. Hence, the fibers are irreducible of dimension $3\cdot 8-10=14$. 

Let $J$ be the fiber over
\begin{equation*}
  I = \Ann(x^4 + y^4, z, w) = (\alpha^4 - \beta^4, \alpha\beta, \gamma \alpha, \gamma \beta,\gamma^2,
  \delta \alpha, \delta \beta, \delta^2, \gamma \delta)
\end{equation*}
in the component \ref{l4+l'4}, which must be generated by $\mathfrak{m}^5$ and the generators
\begin{align*}
    &\hspace{5cm} \theta_0 = \alpha\beta +  a_0 \alpha^3 + b_0 \beta^3 + c_0 \alpha^4 \\
    & \theta_1=\gamma \alpha + a_1 \alpha^3 + b_1 \beta^3 + c_1 \alpha^4  &\quad \theta_2=\delta \alpha +a_2 \alpha^3 + b_2 \beta^3 + c_2 \alpha^4 \\
    &\theta_3=\gamma\beta + a_3 \alpha^3 + b_3 \beta^3+ c_3 \alpha^4 &\quad \theta_4=\delta \beta + a_4 \alpha^3 + b_4 \beta^3 + c_4 \alpha^4  \\
    &\theta_5=\gamma^2 +a_5 \alpha^3 + b_5 \beta^3 + c_5 \alpha^4 &\quad \theta_6=\delta^2 +a_6 \alpha^3 + b_6 \beta^3 + c_6 \alpha^4 \\
    &\hspace{5cm} \theta_7=\gamma\delta +a_7 \alpha^3 + b_7 \beta^3 + c_7 \alpha^4 \\
    &\hspace{5cm}  \zeta = \alpha^4 - \beta^4\text{.}
\end{align*}
In this case, the linear syzygies involving $\alpha$ or $\beta$ are $\gamma \theta_0 - \beta \theta_1, \delta \theta_0 - \beta \theta_2, \gamma \theta_0 - \alpha \theta_3, \delta \theta_0 - \alpha \theta_4, \beta \theta_1 - \alpha \theta_3, \gamma \theta_1 - \alpha \theta_5, \gamma \theta_3 - \beta \theta_5, \delta \theta_2 - \alpha \theta_6, \delta \theta_4 - \beta \theta_6, \beta \theta_2 - \alpha \theta_4, \delta \theta_1 - \alpha \theta_7, \delta \theta_3 - \beta \theta_7, \gamma \theta_2 - \alpha \theta_7, \gamma \theta_4 - \beta \theta_7$ implying $a_3, a_4, a_5, a_6, a_7, b_1, b_2, b_5,b_6, b_7 = 0$. Therefore, the fibers are irreducible of dimension $24 - 10 = 14$.\\
The statement on the components follows from a direct dimension count using Proposition~\ref{14221} and adding up the dimension of each locus with the dimension of the fibers above.
\end{proof}
\begin{proposition}\label{Sm14221}
    Every local algebra with the Hilbert function $H=(1,4,2,2,1)$ is smoothable.
    
\end{proposition}
\begin{proof}
By Proposition \ref{fiberdim}, $\HilbFunc{H}{4}$ has two components $P$ and $Q$. We note that the component $Q$ has a non-empty intersection with the Gorenstein locus as it contains an ideal $\Ann(x^4+y^4+zw)$ in the fiber over $\Ann(x^4+y^4,z,w)$. Since the Gorenstein locus is an open subset of $\Hilb_{10}(\mathbb{A}^n)$, Theorem \ref{Gorenstein} implies that the whole component is inside $\Hilb_{10}^{sm}(\mathbb{A}^n)$.

We show that component $P$ is inside the smoothable locus. We consider the ideal
$$ I=(\delta^2-\alpha^4, \gamma \delta-\beta^3, \beta\delta, \alpha\delta, \gamma^2, \beta\gamma, \alpha\gamma, \alpha\beta, \beta^4)+\mathfrak{m}^5$$
in the fiber above the point $\Ann(x^4,y^3,z,w)$.
Setting
\begin{align*}
    J_1&=(\alpha+1,\beta,\gamma,\delta)\\
    J_2&=(\delta^2-\alpha^3, \gamma\delta-\beta^3, \beta\delta, \alpha\delta, \gamma^2, \beta\gamma, \alpha\gamma, \alpha\beta, \beta^4)
\end{align*}
the ideal $J_2$ has Hilbert function $(1,4,2,2)$ and hence it is smoothable by Proposition \ref{Hf:1n22}. Moreover, $I=\inn_{(1,1,1,2)}(J_1\cap J_2)$ and by a direct check $I$ is a smooth point of the Hilbert scheme $\Hilb_{10}(\mathbb{A}^4)$, which implies that the whole component $P$ containing $I$ is smoothable.  
\end{proof}

\section{Ray degeneration method}\label{RayMethod}
In this section, we exploit the ray degeneration method from \cite{casnati_jelisiejew_notari} to show that algebras with Hilbert functions of the form $(1,H(1),\ldots,H(c),1,\ldots,1)$ having at least $c$ $1$'s at the tail belong to a non-elementary component of $\Hilb_d(\mathbb{A}^n)$. Moreover, all those algebras with socle dimension $2$ are smoothable.

In this section, we use the notation $f^{(s)} = \frac{f^s}{s!}$ for $s \geq 1$ and $f \in S$. The formula for the derivation action becomes
\begin{equation*}
    \alpha^\mathbf{a}\lrcorner  x^{(\mathbf{b})}=\begin{cases}
        x^{(\mathbf{b}-\mathbf{a})} & \mathbf{b}\geq  \mathbf{a},\\
        0 & \mbox{otherwise},
    \end{cases}
\end{equation*}
where $\mathbf{a}, \mathbf{b} \in \mathbb{N}^n$.

Let $\cA$ be an algebra with Hilbert function $(1,H(1), H(2),\ldots, H(c),1,1,\ldots,1)$ of socle degree $s\geq 2c$, so that the inverse system of $\cA$ is generated by a polynomial of degree $s$ and a vector space of polynomials $W$ with $\deg(h)\leq c$ for every $h \in W$. By \cite[Example 4.4]{casnati_jelisiejew_notari}, after a change of coordinates, we can write $\cA = R/\Ann(x_1^{(s)} + g, W)$, where $\alpha_1^c \lrcorner g = 0$ and $\deg g \leq c+1$. Since $s-c \geq c$, we have
\begin{equation*}
    \alpha_1^{s-c}\lrcorner (x_1^{(s)} +g) = x_1^{(c)} + (\alpha_1^{s-c} \lrcorner g) = x_1^{(c)}\text{.}
\end{equation*}
Therefore we may use $x_1^{(c)}$ to modify $h \in W$ linearly, so that $\alpha_1^c \lrcorner h = 0$ for each $h \in W$. 

Let $I = \Ann(x_1^{(s)} + g, W)$ and for $J = \mathfrak{p}_1 \cap I$, let $
  I = J + (\alpha_1^{\nu} - q)$
be the ray decomposition of $I$, where $q\in \mathfrak{p}_1$ (see Subsection~\ref{subsection_ray_family}).
\begin{lemma}\label{a1g=0}
    With the above notation, we have $c+1 \leq \nu \leq s$ and $\alpha_1^{\nu-1}\lrcorner g=0$.
\end{lemma}
\begin{proof}
    First, we prove that $s \geq \nu$. Suppose that $\nu > s$, then $\alpha_1^\nu \lrcorner \langle x_1^{(s)} + g, W\rangle = 0$, and therefore $I = J + (q)$, so $I \subseteq \mathfrak{p}_1$, which is impossible as $R/I$ is a finite dimensional vector space over $\mathbb{K}$.
    
    Now we prove $\alpha_1^{\nu-1}\lrcorner g=0$. Since $(\alpha_1^\nu - q)\lrcorner (x_1^{(s)} + g) =0$, we have 
    \begin{equation*}
        q \lrcorner g = q\lrcorner (x_1^{(s)} + g )= \alpha_1^{\nu} \lrcorner (x_1^{(s)} + g) = x_1^{(s - \nu)} + \alpha_1^\nu \lrcorner g\text{.}
    \end{equation*}
    Then $\alpha_1^{s-\nu}\lrcorner (q \lrcorner g) = \alpha_1^{s-\nu}\lrcorner x_1^{(s-\nu)} + \alpha_1^s\lrcorner g = 1$. Hence, $\alpha_1^{s-\nu} \lrcorner g \neq 0$ and  $s-\nu\leq c -1$. Thus, $\nu -1 \geq s-c \geq c$ and $\alpha_1^{\nu -1 }\lrcorner g = 0$.
\end{proof}
In the following lemmas, we analyze the (upper) ray family attached to the ray decomposition $I = J + (a_1^\nu - q)$. The ideal of the fiber of this family above $t = \lambda$ is $J + (\alpha_1^{\nu}-\lambda \alpha_1^{\nu-1}-q$). In the three following lemmas, we prove that this ideal is equal to
\begin{equation*}
   (\alpha_1 - \lambda, \alpha_2,\ldots, \alpha_n) \cap \Ann( (x_1 + \lambda^{-1})^{(s-1)} -\lambda g, W)\text{.}
\end{equation*}

\begin{lemma}\label{firstinclusion}
    With the above notation, for any $\lambda\neq 0$, we have
\begin{equation*}
    J + (\alpha_1^\nu - \lambda \alpha_1^{\nu-1} -q) \subseteq (\alpha_1-\lambda,\alpha_2,\dots,\alpha_n) \cap ((\lambda^2 \alpha_1^{\nu-1} + \lambda
    q +p)+J)\text{,}
  \end{equation*}
  where $
    p = q \alpha_1 \left(1 + \frac{\alpha_1}{\lambda}+ \dots + \frac{\alpha_1^t}{\lambda^t}\right)\text{,}
  $
  and $t$ is the least natural number such that $\alpha_1^{t+2} \cdot q \in J$ and $t+1 \geq s-\nu$.
\end{lemma}
\begin{proof}
     Since $J, (q) \subseteq \mathfrak{p}_1$, it is enough to prove that
  \begin{equation*}
    \alpha_1^\nu - \lambda \alpha_1^{\nu-1} -q \in J + (\lambda^2 \alpha_1^{\nu -1} +\lambda q + p)\text{.}
  \end{equation*}
   Modulo $J + (\lambda^{2} \alpha_1^{\nu-1} + \lambda q + p)$, we have
  $$\lambda^2(\alpha_1^\nu - \lambda \alpha_1^{\nu-1} -q) \equiv \lambda^2(\alpha_1^{\nu} + \frac{p}{\lambda}) 
     \equiv \alpha_1(-\lambda q - p) + \lambda p 
     = p(\lambda -\alpha_1) - \alpha_1 q \lambda.$$
It suffices to prove that $p(\lambda-\alpha_1) - \alpha_1 q \lambda \in J$. We have
  \begin{align*}
    p \lambda\left(1 -\frac{\alpha_1}{\lambda}\right) &= \lambda q \alpha_1\left(1-\frac{\alpha_1}{\lambda}\right)\left(1+\frac{\alpha_1}{\lambda}
    +\dots+ \frac{\alpha_1^t}{\lambda^t}\right) \\
	&= \lambda q \alpha_1\left(1-\frac{\alpha_1^{t+1}}{\lambda^{t+1}}\right) \equiv \lambda q \alpha_1 \pmod J\text{,}
  \end{align*}
  as desired.
\end{proof}
\begin{lemma}\label{2ndinclusion} With the above notation, for every $\lambda\neq 0$, we have
    \begin{equation*}
  (\lambda^2 \alpha_1^{\nu -1} + \lambda q + p) + J \subseteq \Ann((x_1 + \lambda^{-1})^{(s-1)} - \lambda g,W)\text{.}
  \end{equation*}
\end{lemma}
\begin{proof}
     Since $(\alpha_1^\nu - q) \hook (x_1^{(s)} + g) = 0$, we have $x_1^{(s-\nu)} = q \hook g$. Moreover, as $J \hook
  x_1^{(s)} = 0$ (because $J \subseteq \mathfrak{p}_1$) and $J \hook (x_1^{(s)} + g) = 0$, we get $J \hook g = 0$. This implies that
  \begin{equation*}
    J \subseteq \Ann(( x_1 + \lambda^{-1})^{(s-1)} - \lambda g,W)\text{.}
  \end{equation*}
Moreover, we have 
  \begin{align*}
    (\lambda q + p) \hook \lambda g &= \left(q\lambda\left(1 + \frac{\alpha_1}{\lambda}
    +\dots+\frac{\alpha_1^{t+1}}{\lambda^{t+1}}\right)\right) \hook (\lambda g) \\
    &= \lambda^2\left(1 + \frac{\alpha_1}{\lambda} +\dots+\frac{\alpha_1^{t+1}}{\lambda^{t+1}}\right) \hook x_1^{(s-\nu)} \\
    &= \lambda^2\left(x_1^{(s-\nu)} + \frac{x_1^{(s-\nu-1)}}{\lambda} + \dots + \frac{1}{\lambda^{s-\nu}}\right)= \lambda^2(x_1+\lambda^{-1})^{(s-\nu)}\text{.}
  \end{align*}
This together with Lemma \ref{a1g=0}, implies that
  \begin{equation*}
    (\lambda^2 \alpha_1^{\nu-1} + \lambda q + p)\hook ((x_1 + \lambda^{-1})^{(s-1)} - \lambda g) = 0\text{,}
  \end{equation*}
  and since $\alpha_1^{\nu-1}\lrcorner h=q\lrcorner h=p\lrcorner h=0$ for $h\in W$, we get
  \begin{equation*}
     \lambda^2 \alpha_1^{\nu-1} + \lambda q + p \in \Ann(( x_1 + \lambda^{-1})^{(s-1)} - \lambda g,W)\text{.}
  \end{equation*}
\end{proof}
\begin{lemma}
\label{ref:hilbert_function:lemma}
The Hilbert function of $R/\Ann((x_1 + \lambda^{-1})^{(s-1)} - \lambda g, W)$ is
 
 \begin{equation*}
     (1,H(1),\dots,H(c),1,\dots,1)
 \end{equation*}
with socle degree $s-1$. 
\end{lemma}
\begin{proof}
  
  For sets $X \subseteq R$, and $U \subseteq S$, we denote by $\Diff_X(U)$ the linear span
  \begin{equation*}
    \langle \theta \lrcorner f \ :\ \theta \in X \text{ and } f \in U \rangle\text{.}
  \end{equation*}
  We also write $\Diff(U) = \Diff_R(U)$.
  Consider the sets:
  \begin{align*}
      Y &= \{1, \alpha_1, \alpha_1^2,\ldots \}, \\
      Y' &= \{ m \in \mathbb{K}[\alpha_2,\ldots,\alpha_n] \ :\ m \text{ is a monomial, } \deg m \geq 1\}.
  \end{align*}
  Let us define an affine subspace of $S$ by $\mathbb{A} = g+ W$. We get that
  \begin{equation*}
      \Ann(x_1^{(s)}+ g, W) = \Ann(x_1^{(s)}+ \mathbb{A}).
  \end{equation*}
  We have the following decomposition
  \begin{align*}
    \Diff(x_1^{(s)} + \mathbb{A}) &= \Diff_{Y}(x_1^{(s)} + \mathbb{A})+  \sum_{i=0}^{c-1} \Diff_{\alpha_1^i \cdot Y'}(x_1^{(s)} +\mathbb{A}) \\
    &= \Diff_Y(x_1^{(s)} + g) + \Diff_Y(W) + \sum_{i=0}^{c-1} \Diff_{\alpha_1^i \cdot Y'}(x_1^{(s)} +\mathbb{A})\text{.}
  \end{align*} We calculate:
  \begin{align*}
      \Diff_Y(x_1^{(s)} + g) =  \langle & x_1^{(s)} + g, \\
                        &x_1^{(s-1)} + \alpha_1 \lrcorner g, \\
                        & \vdots \\
                        & x_1^{(s-c+1)} + \alpha_1^{c-1} \lrcorner g, \\
                        & x_1^{(s-c)}, x_1^{(s-c-1)},\ldots, x_1, 1 \rangle\text{,} 
  \end{align*}
  and
  \begin{equation*}
      \Diff_{\alpha_1^i\cdot Y'}(x_1^{(s)} + \mathbb{A}) = \Diff_{Y'}( \alpha^i_1 \lrcorner \mathbb{A} )
  \end{equation*}
  for $i = 0,1,\ldots,c-1$. 
  Similarly, we have
  \begin{equation*}
    \Diff((x_1+\lambda^{-1})^{(s-1)} - \lambda \mathbb{A}) = \Diff_Y((x_1+\lambda^{-1})^{(s-1)} - \lambda g) + \Diff_Y(W) +  \sum_{i=0}^{c-1} \Diff_{Y'}(\alpha_1^i \lrcorner \mathbb{A})\text{.}
  \end{equation*}
  We calculate:
  \begin{align*}
      \Diff_Y((x_1 + \lambda^{-1})^{(s-1)}- \lambda g) = \langle & (x_1 + \lambda^{-1})^{(s-1)} - \lambda g, \\
                        & (x_1 + \lambda^{-1})^{(s-2)} - \lambda \alpha_1 \lrcorner g, \\
                        & \vdots \\
                        & (x_1 + \lambda^{-1})^{(s-c)} - \lambda\alpha_1^{c-1} \lrcorner g, \\
                        & (x_1 + \lambda^{-1})^{(s-c-1)}, (x_1 + \lambda^{-1})^{(s-c-2)},\ldots, x_1 + \lambda^{-1}, 1 \rangle\text{.}
  \end{align*}
  For any $F \in S$, we have a filtration on $\Diff(F)$ given by 
  \begin{equation*}
     \Diff^{\leq i}(F) = \{ h \in \Diff(F) \ :\ \deg h \leq i\}. 
  \end{equation*}
  The Hilbert function of $R/\Ann(F)$ is given by the formula
  \begin{equation*}
    H_{R/\Ann(F)}(i) = \dim_{\mathbb{K}} \frac{\Diff^{\leq i}(F)}{\Diff^{\leq i-1}(F)}\text{.}
  \end{equation*}
  We compare the Hilbert functions of $R/\Ann(x_1^{(s)} +g , W)$ and $R/\Ann((x_1 + \lambda^{-1})^{(s-1)} -\lambda g, W)$. From the above formulas it follows that the Hilbert functions are the same up to degree $c$, and from degree $c+1$ on they are equal to $1$, and the number of $1$'s is one higher in $R/\Ann(x_1^{(s)} + g, W)$. Since the Hilbert function of $R/\Ann(x_1^{(s)} + g, W)$ is given by assumption, the statement of the Proposition follows.
\end{proof}
\begin{proposition}
    \label{ref:flat:prop}
    Let $s$ and $c$ be as above.
    Let $\langle x_1^{(s)} + g, W\rangle$ be a subspace of polynomials of degree at most $s$
    such that $\alpha_1^c \lrcorner g = 0$ and $\alpha_1^c \lrcorner h=0 $ for each $h \in W$. Then the ray decomposition $\Ann(x_1^{(s)} + g, W) = (\alpha_1^\nu - q) + J$, where $J = \Ann(x_1^{(s)} + g,W)\cap (\alpha_2,\dots, \alpha_n)$ gives rise to an upper ray degeneration.
\end{proposition}
\begin{proof}
    First, we claim that for every $\gamma\in \mathfrak{p}_1$, we have
    $$
        \gamma \cdot (\alpha_1^{\nu} - t \alpha_1^{\nu -1} -q) \in J[t]\text{.}
    $$
    Note that $(\alpha_1^\nu-q) \lrcorner \langle x_1^{(s)} + g,W\rangle = 0$ and by Lemma \ref{a1g=0}, $\alpha_1^{\nu-1}\gamma \lrcorner (x_1^{(s)} + g) =\alpha_1^{\nu -1} \gamma\lrcorner g = 0$. Moreover, since $\nu - 1\geq c$, for any $h \in W$ we have $\alpha_1^{\nu -1} \gamma \lrcorner h = 0$. This implies that $\alpha_1^{\nu -1}\gamma \in J$ and since $(\alpha_1^\nu - q)\gamma \in J$, the claim holds.

    Let $\mathcal{J} \subseteq R[t]$ be the ideal defining the upper ray family and $\mathcal{J}_0 := \mathcal{J}\cap R$.  By  \cite[Proposition 2.12]{casnati_jelisiejew_notari}, it is sufficient to show that for any $\lambda\in \mathbb{K}$, we have $(t-\lambda)\cap\mathcal{J}\subset (t-\lambda)\mathcal{J}+\mathcal{J}_0[t]$. Let $\theta \in \mathcal{J} \cap (t-\lambda)$ be an element and write $\theta = \theta_1 + \theta_2 (\alpha_1^\nu - t\alpha_1^{\nu-1} - q)$ . By \cite[Remark 2.13]{casnati_jelisiejew_notari}, we may assume that $\theta_1 \in J$ and $\theta_2 \in R$. Since $\theta \in (t-\lambda)$, we have $\theta_1+\theta_2(\alpha_1^\nu - \lambda \alpha_1^{\nu -1} -q) = 0$. Since $\theta_1 \in \mathfrak{p}_1$, we must have $\theta_2 \in \mathfrak{p}_1$ and by the claim above, we have $\theta_2 (\alpha_1^\nu - t\alpha_1^{\nu-1} -q) \in J[t]\subseteq \mathcal{J}_0[t]$. Since $\theta_1 \in J \subseteq \mathcal{J}_0[t]$, we get that $\theta\in (t-\lambda)\mathcal{J}+\mathcal{J}_0[t]$ and the upper ray family is flat.
\end{proof}
\begin{theorem}\label{rayflat}
    Let $\cA$ be a local algebra with Hilbert function $(1,H_1, H_2,\ldots, H_c,1,1,\ldots,1)$ of socle degree $s\geq 2c$. Then the spectrum of $\cA$ deforms to a disjoint union of a point and the spectrum of a local algebra with Hilbert function $(1,H_1, H_2,\ldots, H_c, 1, 1,\ldots,1)$ and socle degree $s-1$ (the number of ones is one less).  
\end{theorem}
\begin{proof}
We use the upper ray degeneration described in Proposition \ref{ref:flat:prop}.
By Lemmas \ref{firstinclusion} and \ref{2ndinclusion}, for every $\lambda\neq 0$, we have
    \begin{equation}\label{inclu}
  J + (\alpha_1^\nu - \lambda \alpha_1^{\nu -1} -q)\subseteq
  (\alpha_1 -\lambda, \alpha_2,\dots,\alpha_n) \cap \Ann( (x_1 + \lambda^{-1})^{(s-1)} -\lambda g,W)
\end{equation}
Since the algebra corresponding to $\Ann( (x_1 + \lambda^{-1})^{(s-1)} -\lambda g,W)$ has the Hilbert function $(1,H_1, H_2,\ldots, H_c,1,1,\ldots,1)$ of socle degree $s-1$, the two ideals in Equation~\eqref{inclu} have the same colength, and we have the equality proving the statement.
\end{proof}
\begin{remark}
  We remark that in Theorem~\ref{rayflat}, an algebra $\cA$ of socle dimension $\tau$ is described as the limit of a union of a point and another  algebra of a lower degree, yet with socle dimension as $\cA$.  
\end{remark}
 Now using Proposition \ref{Hf:1n21}, and Proposition \ref{smooth1nr1} which will be established independently later, as a conclusion of Theorem \ref{rayflat}, we obtain that:
\begin{corollary}\label{rayflat_socledim2}
    Every local algebra with the Hilbert function $$(1,4,2,1,1), (1,4,3,1,1), (1,4,2,1,1,1),(1,5,2,1,1)$$ and socle dimension 2 is smoothable.
\end{corollary}

\section{Components of Hilbert schemes}\label{Components}
In this section, we determine the components of $\Hilb_d(\mathbb{A}^n)$ for $d=9,10$.

\subsection{Local algebras $(\cA,\mathfrak{n})$ with $\mathfrak{n}^3=0$}\label{Sectionn3=0}

This subsection considers the local algebras $(\cA,\mathfrak{n})$  with $\mathfrak{n}^3 = 0$. We will show that such an algebra belongs to a non-elementary component of $\Hilb_d(\mathbb{A}^n)$ if $d\in\{9,10\}$, unless the algebra has Hilbert function $(1,5,3)$ or $(1,6,3)$. We first note that all algebras with $\mathfrak{n}^2=0 $ are smoothable, see e.g \cite[Proposition 4.15]{ccvv}. Therefore, here we consider algebras with $\mathfrak{n}^2\neq 0$ and the Hilbert function $(1,n,r)$ with $d=r+n+1$. By \cite[Proposition 4.2]{ccvv}, the locus of all local algebras with such a Hilbert function is irreducible.\\

Let $A=(A_1,\ldots ,A_n)$ be an $n$-tuple of commuting $d\times d$ matrices corresponding to the algebra $\cA$. Recall that the matrices $A_i$ represent multiplications with the variables $\alpha_i$ on the algebra $\cA$. Since $H=(1,n,r)$, we may choose a $\KK$-basis of $\cA$ of the form $\{q_1,\ldots ,q_r,l_1,\ldots ,l_n,1\}$ for some linear forms $l_1,\ldots ,l_n$ and quadratic forms $q_1,\ldots ,q_r$. In this basis the matrices $A_i$ are of the form
\begin{equation}\label{matrices:cube_zero}
A_i=\begin{bmatrix}
    0 & B_i & 0 \\
    0 & 0 & c_i \\
    0 & 0 & 0
\end{bmatrix},
\end{equation}
with blocks of respective sizes $r=d-n-1$, $n$ and 1, for some vectors $c_i\in \KK^n$ and some matrices $B_i\in \mathbb{M}_{r\times n}$ such that $c_1,\ldots ,c_n$ are linearly independent and the common cokernel of $B_1,\ldots ,B_n$ is trivial. The conditions on the common cokernels are consequences of the equalities
$$\dim_{\KK}\bigcap_{i=1}^n\ker A_i^T=\dim_{\KK}\cA/\mathfrak{n}=1$$
and
$$\dim_{\KK}\bigcap_{i,j}\ker (A_iA_j)^T=\dim_{\KK}\cA/\mathfrak{n}^2=n+1,$$
see Remark \ref{3.20}.

By Theorem \ref{Shafarevich} we know that algebras with Hilbert functions $(1,n,r)$ are smoothable for $r\le 2$, while they form generically reduced elementary components of $\Hilb_d(\mathbb{A}^n)$ if $3\le r\le \frac{(n-1)(n-2)}{6}+2$ and $(n,r)\ne (5,4)$. In the cases $d=9,10$, we therefore get elementary components parametrizing algebras with Hilbert functions $(1,5,3)$ and $(1,6,3)$. Using Theorem \ref{11points}, it follows that we have to consider only algebras with Hilbert functions $(1,4,4)$, $(1,4,5)$ and $(1,5,4)$. We show that algebras with these Hilbert functions are smoothable. Recall that the case $(1,5,4)$ was an exception in Shafarevich's result remarked in \ref{exception}.

\begin{proposition}\label{HF:144,154}
    Local algebras with Hilbert functions $(1,4,4)$ and $(1,5,4)$ are smoothable.
\end{proposition}

\begin{proof}
    The subschemes $\HilbFunc{(1,n,r)}{n}$ are irreducible, so for $r=4$ and $n\in \{4,5\}$ it suffices to find a smoothable algebra in $\HilbFunc{(1,n,r)}{n}$ which is a smooth point on the Hilbert scheme $\Hilb_{1+n+r}(\mathbb{A}^n)$. We will prove the corresponding statement for quadruples of matrices, i.e. we will find quadruples of commuting matrices of the form \eqref{matrices:cube_zero} which are smooth points on the principal components of $\CommMat{9}{4}$, respectively $\CommMat{10}{4}$.
    
    We start with the case $n=4,r=4$. Let
    $$A_1(\lambda)=\begin{bmatrix}
        \lambda E & \mathbf{I} & 0\\
        0 & \lambda F & e_1\\
        0 & 0 & 0
    \end{bmatrix}\quad \mathrm{where}\quad E=\begin{bmatrix}
        0 & 0 & 0 & 1\\
        1 & 0 & 0 & 0\\
        0 & 1 & 0 & 0\\
        0 & 0 & 1 & 0
    \end{bmatrix}, \, F=\begin{bmatrix}
        0 & 0 & 0 & -1\\
        1 & 0 & 0 & 0\\
        0 & 1 & 0 & 0\\
        0 & 0 & 1 & 0
    \end{bmatrix},$$
    $e_1$ is the first standard basis vector of $\KK^4$ and $\mathbf{I}$ denotes the identity matrix. We define 
    \begin{align*}
        &A_2(\lambda)=\frac{1}{\lambda}A_1(\lambda)^2-\frac{1}{3\lambda^5}A_1(\lambda)^6,\\
        &A_3(\lambda)=\frac{1}{\lambda^2}A_1(\lambda)^3-\frac{1}{\lambda^6}A_1(\lambda)^7,\\
        &A_4(\lambda)=\frac{1}{\lambda^3}A_1(\lambda)^4-\frac{3}{\lambda^7}A_1(\lambda)^8.
    \end{align*}
    It is clear that the matrices $A_i(\lambda)$ commute for $\lambda\ne 0$ and that $A_1(\lambda)$ has 9 distinct eigenvalues. Therefore the quadruple $(A_i(\lambda))_{i=1}^4$ belongs to the principal component of $\CommMat{9}{4}$ for each $\lambda\ne 0$. Moreover, the constants in the expressions for $A_i(\lambda)$ are chosen in such a way that $A_i(\lambda)$ is of the form $\begin{bmatrix}
        \lambda E_i & B_i & 0\\
        0 & \lambda F_i & c_i\\
        0 & 0 & 0
    \end{bmatrix}$, so it makes sense to define $A_i(0)$, and $(A_i(0))_{i=1}^4$ is also a commuting quadruple of matrices and it belongs to the principal component. It is also easy to see that the quadruple $(A_i(0))_{i=1}^4$ corresponds to an algebra from $\HilbFunc{(1,4,4)}{4}$. Moreover, a straightforward computation using Lemma \ref{tangent_space_commuting_marices} shows that the dimension of the tangent space to $\CommMat{9}{4}$ at $(A_i(0))_{i=1}^4$ is equal to 108 which is the dimension of the principal component. So the point $(A_i(0))_{i=1}^4$ is smooth on the principal component, as required.

    The case $n=5,r=4$ is proved in exactly the same way, starting from
    $$A_1(\lambda)=\begin{bmatrix}
        \lambda E & B & 0\\
        0 & \lambda F & e_1\\
        0 & 0 & 0
    \end{bmatrix}\, \mathrm{where}\, E=\begin{bmatrix}
        0 & 0 & 0 & -1\\
        1 & 0 & 0 & 0\\
        0 & 1 & 0 & 0\\
        0 & 0 & 1 & 0
    \end{bmatrix}, \, F=\begin{bmatrix}
        0 & 0 & 0 & 0 & 1\\
        1 & 0 & 0 & 0 & 0\\
        0 & 1 & 0 & 0 & 0\\
        0 & 0 & 1 & 0 & 0\\
        0 & 0 & 0 & 1 & 0
    \end{bmatrix}, \, B=\begin{bmatrix}
        0 & 1 & 0 & 0 & 0\\
        0 & 0 & 1 & 0 & 0\\
        0 & 0 & 0 & 1 & 0\\
        1 & 0 & 0 & 0 & 1
    \end{bmatrix}$$
    and setting
    \begin{align*}
        &A_2(\lambda)=\frac{1}{\lambda}A_1(\lambda)^2+\frac{1}{3}A_1(\lambda)-\frac{1}{3\lambda^5}A_1(\lambda)^6,\\
        &A_3(\lambda)=\frac{1}{\lambda^2}A_1(\lambda)^3,\\
        &A_4(\lambda)=\frac{1}{\lambda^3}A_1(\lambda)^4-\frac{1}{3\lambda^2}A_1(\lambda)^3+\frac{1}{3\lambda^7}A_1(\lambda)^8,\\
        &A_5(\lambda)=\frac{1}{\lambda^4}A_1(\lambda)^5-\frac{2}{3\lambda^3}A_1(\lambda)^4+\frac{2}{3\lambda^8}A_1(\lambda)^9.
    \end{align*}
\end{proof}

In the case of Hilbert function $(1,4,5)$ we were unable to find a point in $\HilbFunc{H}{4}$ which is smooth in $\Hilb_{10}(\mathbb{A}^4)$, and we prove smoothability in a different way.

\begin{proposition}\label{HF:145}
    Local algebras with Hilbert function $(1,4,5)$ are smoothable.
\end{proposition}

\begin{proof}
    We will prove the matrix version of the proposition, i.e. we will show that all quadruples of commuting matrices $(A_1,\ldots ,A_4)$ of the form \eqref{matrices:cube_zero} belong to the principal component of $\CommMat{10}{4}$. To do this, let $\mathcal{V}$ be the set of all quadruples $(A_1,A_2,A_3,A_4)$ of commuting matrices of the form \eqref{matrices:cube_zero}
    with blocks of sizes $5,4,1$, where $c_1,c_2,c_3,c_4$ are linearly independent and $B_1,B_2,B_3,B_4$ have trivial common cokernel. We consider $\mathcal{V}$ as an open subscheme of the scheme of commuting matrices of the shape $\begin{bmatrix}
        0 & \ast & 0\\
        0 & 0 & \ast\\
        0 & 0 & 0
    \end{bmatrix}$, so the tangent space to $\mathcal{V}$ is given by Lemma \ref{tangent_space_commuting_marices} and the linear equations describing the shape of the matrices. As in the proof of \cite[Lemma 6.7]{JS} we see that the quadruples $(c_1,\ldots ,c_4)$, where $(A_1,\ldots ,A_4)\in \mathcal{V}$, form an open subset of $(\KK^4)^4$, and that given a fixed quadruple $(c_1,\ldots ,c_4)$ the matrices $B_1,\ldots ,B_4$ are solutions of a system of linearly independent $5\cdot{4\choose 2}$ linear equations, so $\dim \mathcal{V}=4\cdot 4+4\cdot 4\cdot 5-5\cdot {4\choose 2}=66$. 

    Let $\mathcal{Z}$ be the scheme of quadruples of commuting matrices of the form $$A_i=\begin{bmatrix}
        E_i & B_i & 0\\
        0 & F_i & c_i\\
        0 & 0 & 0
    \end{bmatrix}$$ and let $\mathcal{U}$ be its open subscheme where $c_1,c_2,c_3,c_4$ are linearly independent, the common cokernel of $B_1,B_2,B_3,B_4$ is trivial and $A_1$ has 10 distinct eigenvelues. Again, the tangent space to $\mathcal{U}$ is given by Lemma \ref{tangent_space_commuting_marices} and the shape of the matrices. A matrix that commutes with a matrix having all eigenvalues distinct is a polynomial in that matrix, so the dimension of the principal component of the scheme of quadruples of commuting block upper triangular matrices with blocks of sizes $5,4,1$ and lower right corner equal to 0 is $10^2-4\cdot 5-10+3\cdot 9=97$. As $\mathcal{U}$ is an open subset of the intersection of this scheme with the space of quadruples of matrices with upper right block corner equal to 0, the dimension of each of its irreducible components is at least $97-4\cdot 5=77$.

    Let 
    $$A_1=\begin{bmatrix}
        E_1 & B_1 & 0\\
        0 & F_1 & e_1\\
        0 & 0 & 0
    \end{bmatrix},\ \text{with}\ E_1=\begin{bmatrix}
        0 & 0 & 0 & 0 & 1\\
        1 & 0 & 0 & 0 & 0\\
        0 & 1 & 0 & 0 & 0\\
        0 & 0 & 1 & 0 & 0\\
        0 & 0 & 0 & 1 & 0
    \end{bmatrix}, B_1=\begin{bmatrix}
        1 & 0 & 0 & 0\\
        0 & 1 & 0 & 0\\
        0 & 0 & 1 & 0\\
        0 & 0 & 0 & 1\\
        0 & 0 & 0 & 0
    \end{bmatrix}, F_1=\begin{bmatrix}
        0 & 0 & 0 & -1\\
        1 & 0 & 0 & 0\\
        0 & 1 & 0 & 0\\
        0 & 0 & 1 & 0
    \end{bmatrix}$$
    and $e_1$ is the first standard basis vector. Furthermore, let
    \begin{align*}
        &A_2=A_1^2-\frac{3}{32}A_1^5-\frac{1}{4}A_1^7-\frac{1}{8}A_1^8-\frac{3}{32}A_1^9,\\
        &A_3=A_1^3-\frac{3}{8}A_1^5-\frac{1}{2}A_1^8-\frac{3}{8}A_1^9,\\
        &A_4=A_1^4-\frac{3}{4}A_1^5-\frac{3}{4}A_1^9.
    \end{align*}
    The constants are chosen in such a way that $(A_1,A_2,A_3,A_4)\in \mathcal{U}$, i.e. that the upper right block corners of the matrices $A_i$ are zero. Let $\mathcal{C}$ be the component of $\overline{\mathcal{U}}$ containing $(A_1,A_2,A_3,A_4)$. Then $\mathcal{C}$ is also a component of $\mathcal{Z}$. We can compute that $$\dim T_{(A_1,A_2,A_3,A_4)}\mathcal{C}=\dim T_{(A_1,A_2,A_3,A_4)}\mathcal{Z}=77,$$ so $\dim \mathcal{C}=77$ and $(A_1,A_2,A_3,A_4)$ is a smooth point on $\mathcal{C}$.

    Now let $\pi\colon \mathcal{C}\to \overline{\mathcal{V}}$ be the projection that sends the diagonal blocks of matrices to zero, and let $(A_1',A_2',A_3',A_4')=\pi(A_1,A_2,A_3,A_4)$. A computation shows that $(A_1',A_2',A_3',A_4')$ is a smooth point on $\overline{\mathcal{V}}$ and that the differential $$d\pi\colon T_{(A_1,A_2,A_3,A_4)}\mathcal{C}\to T_{(A_1',A_2',A_3',A_4')}\mathcal{V},$$
    which is also the projection that sends the diagonal blocks to zero, is surjective. The map $\pi$ is therefore dominant.

    The dominance of the map $\pi$ implies that for a general quadruple $(Y_1,Y_2,Y_3,Y_4)\in \mathcal{V}$ with $Y_i=\begin{bmatrix}
        0 & B_i & 0\\
        0 & 0 & c_i\\
        0 & 0 & 0
    \end{bmatrix}$ there exists $(X_1,X_2,X_3,X_4)\in \mathcal{C}$ with 
$$X_i=\begin{bmatrix}
        E_i & B_i & 0\\
        0 & F_i & c_i\\
        0 & 0 & 0
    \end{bmatrix}$$    
    such that $X_1$ has 10 distinct eigenvalues. However, then it is clear that $$\left(\begin{bmatrix}
        \lambda E_i & B_i & 0\\
        0 &\lambda F_i & c_i\\
        0 & 0 & 0
    \end{bmatrix}\right)_{i=1}^4\in \mathcal{C}$$
    for each $\lambda \ne 0$, and that for $\lambda\ne 0$ the first matrix in the quadruple has 10 distinct eigenvalues. So the quadruple belongs to the principal component of $\CommMat{10}{4}$ for all $\lambda\in\KK$ and in particular this holds for $(Y_1,Y_2,Y_3,Y_4)$. Therefore, $\mathcal{V}$ is a subset of the principal component, as required.
\end{proof}

\begin{remark}\label{linkage}
    Joachim Jelisiejew informed us that smoothability of algebras with Hilbert functions $(1,4,4)$ and $(1,4,5)$ could be proved also using linkage. We provide here an outline of the argument.

    The definition of linkage is due to Peskine and Szpiro \cite{liaison}. Two ideals $I,J$ of a regular ring $R$ are {\em linked} if there is a regular sequence $r_1,\ldots ,r_k\in R$ such that $J=((r_1,\ldots ,r_k):I)$ and $I=((r_1,\ldots ,r_k):J)$. To show smoothability of local algebras with Hilbert functions $(1,4,r)$ where $r\in \{4,5\}$ the main idea is to prove that a general ideal $I$ in $\HilbFunc{(1,4,r)}{4}$ is linked to an ideal $J$ in $\Hilb_{11-r}(\mathbb{A}^4)$, and then the theorem by Huneke and Ulrich \cite[Theorem 3.10]{Huneke-Ulrich} implies that smoothability of $R/I$ follows from smoothability of $R/J$.

    More precisely, observe first that if $R/I\in \HilbFunc{(1,4,r)}{4}$, then $I$ is generated by $10-r$ quadrics and all cubics. Since $\HilbFunc{(1,4,r)}{4}$ is irreducible, it suffices to prove smoothability of algebras belonging to some open subset. We may therefore consider such algebras $R/I\in \HilbFunc{(1,4,r)}{4}$ that the ideal $I$ is generated only by quadrics $q_1,\ldots ,q_{10-r}$ and that $q_1,q_2,q_3,q_4$ is a regular sequence. Then $K=(q_1,q_2,q_3,q_4)\lhd R$ is a complete intersection, so the Hilbert function of $R/K$ is $(1,4,6,4,1)$. Let $Z=\Spec(R/I)$ and $Z'=\Spec(R/K)$. Then $Z\subseteq Z'$ is a pair of subschemes of $\mathbb{A}^4$ of degrees $5+r$ and 16, with $Z'$ complete intersection, so it belongs to the sublocus $\Hilb_{\ast,\mathrm{lci}}^{(5+r,16)}(\mathbb{A}^4)$ of the nested Hilbert scheme $\Hilb^{(5+r,16)}(\mathbb{A}^4)$, defined in \cite[Definitions 5.1]{Jelisiejew-Ramkumar-Sammartano}. Now we use Propositions 5.2 and 5.3 of \cite{Jelisiejew-Ramkumar-Sammartano}. By \cite[Proposition 5.3]{Jelisiejew-Ramkumar-Sammartano} we have an isomorphism $\Hilb_{\ast,\mathrm{lci}}^{(5+r,16)}(\mathbb{A}^4)\to \Hilb_{\ast,\mathrm{lci}}^{(11-r,16)}(\mathbb{A}^4)$ that sends $Z\subseteq Z'$ to $Z''\subseteq Z'$ where $Z''$ is the closed subscheme of $\mathbb{A}^4$ defined by $(K:I)$, i.e. by the link of $I$. Since $r\ge 4$, the Hilbert scheme $\Hilb_{11-r}(\mathbb{A}^4)$ is irreducible, hence by \cite[Proposition 5.2]{Jelisiejew-Ramkumar-Sammartano} the preimage of the projection $\Hilb_{\ast,\mathrm{lci}}^{(11-r,16)}(\mathbb{A}^4)\to \Hilb_{11-r}(\mathbb{A}^4)$ is the closure of the locus of tuples of points (see \cite[Definitions 5.1]{Jelisiejew-Ramkumar-Sammartano}). In particular, the pair $Z''\subseteq Z'$ belongs to the closure of the locus of tuples of points, and by the above isomorphism $\Hilb_{\ast,\mathrm{lci}}^{(5+r,16)}(\mathbb{A}^4)\to \Hilb_{\ast,\mathrm{lci}}^{(11-r,16)}(\mathbb{A}^4)$ so does $Z\subseteq Z'$. Now we use \cite[Proposition 5.2]{Jelisiejew-Ramkumar-Sammartano} again to see that $Z$ is smoothable.
\end{remark}

\subsection{Local algebras $(A,\mathfrak{n})$ with $\mathfrak{n}^4=0$}

In this subsection we show that each local algebra $(\mathcal{A},\mathfrak{n})$ with $\mathfrak{n}^4=0$ and $\mathfrak{n}^3\ne 0$ belongs to a non-elementary component of $\Hilb_d(\mathbb{A}^n)$ if $d\in \{9,10\}$. Moreover, we show that such algebra is smoothable if its socle dimension is at most 2. By Theorem \ref{11points} we may assume $n\ge 4$, so the possible Hilbert functions for the algebras under consideration are $(1,n,r,1)$ for $n\ge 4$ and $r\le 8-n$, $(1,n,2,2)$ for $n\in \{4,5\}$ and $(1,4,3,2)$. The algebras with Hilbert functions $(1,n,2,2)$ are smoothable by Proposition \ref{Hf:1n22}, so we consider the other two cases.

\subsubsection{Local algebras with Hilbert function $(1,n,r,1)$}

Let $d,n,r$ be positive integers with $d=n+r+2$. We will consider local algebras $(\cA,\mathfrak{n})$ with Hilbert function $(1,n,r,1)$, supported at zero. Let $(A_1,\ldots ,A_n)$ be an $n$-tuple of nilpotent commuting $d\times d$ matrices corresponding to the algebra $\cA$. Note that the matrices $A_i$ represent multiplications by the variables $\alpha_i$ on the algebra $\cA$. Since $H=(1,n,r,1)$, it is clear that we can find a $\KK$-basis of $\cA$ of the form $\{c,q_1,\ldots ,q_r,l_1,\ldots ,l_n,1\}$ where $c$ is a cubic form, $q_1,\ldots ,q_r$ are quadratic forms and $l_1,\ldots ,l_n$ are linear forms. In this basis the matrices $A_i$ are of the block form
\begin{equation}\label{blocks_m^4}
A_i=\begin{bmatrix}
    0 & b_i^T & c_i^T & 0\\
    0 & 0 & D_i & 0\\
    0 & 0 & 0 & f_i\\
    0 & 0 & 0 & 0
\end{bmatrix}
\end{equation}
with blocks of respective sizes $1,r,n,1$, where $b_i\in \KK^r$,  $c_i, f_i\in \KK^n$ and $D_i\in \mathbb{M}_{r\times n}$, such that $f_1,\ldots ,f_n$ are linearly independent, the common cokernel of $D_1,\ldots ,D_n$ is trivial and not all $b_i$ are zero. As in Section \ref{Sectionn3=0}, the conditions on common cokernels follow from Remark \ref{3.20}.

Assume now that $r\le n$. For each $s\in\{1,\ldots ,r\}$ let $\X{n}{r}{s}$ be the set of all $n$-tuples of commuting matrices $(A_1,\ldots ,A_n)$ that can be in some basis written in the form \eqref{blocks_m^4} such that $f_1,\ldots ,f_n$ are linearly independent, the common cokernel of $D_1,\ldots ,D_n$ is trivial and $\dim_{\KK} \mathrm{Span}\{b_1,\ldots ,b_n\}=s$.

\begin{lemma}\label{irreducibility of X}
    The locally closed locus $\X{n}{r}{s}$ is irreducible if $1\le s\le r\le n$.
\end{lemma}

\begin{proof}
    Note that $\GL_d\times \GL_n$ acts on $\X{n}{r}{s}$ by conjugation and linear change of matrices. It is easy to see that $\X{n}{r}{s}$ is the $\GL_d\times\GL_n$-orbit of the locus of all $n$-tuples of block matrices of the form \eqref{blocks_m^4} satisfying $f_i=e_i$ for $i=1,\ldots ,n$ (where $e_i$ denotes the vector with 1 on the $i$-th component and zeros elsewhere), $b_i=e_{r-i+1}$ for $i=1,\ldots ,s$ and $b_i=0$ for $i>s$. The commutativity in this locus is described by linear equations, so the locus is an affine space, hence irreducible, and the irreducibility of $\X{n}{r}{s}$ also follows.
\end{proof}

\begin{example}\label{ex: minimal socle dimension in X}
    Let $1\le s\le r\le n$. For $i=1,\ldots ,n$, let $A_i$ be a matrix of the form \eqref{blocks_m^4} with $c_i=0$ and $f_i=e_i$ for $i=1,\ldots ,n$, $b_i=e_{r-i+1}$ for $i=1,\ldots ,s$, $b_i=0$ for $i>s$,
    $$D_i=\sum_{j=1}^{s+1-i}e_{r-s+i+j-1}e_j^T+\sum_{j=1}^{r-s}e_je_{n-r+s+1-i+j}^T\quad \mathrm{for}\quad i\le s,$$
    $$D_i=\sum_{j=1}^{r-s}e_je_{n-r+s+1-i+j}^T\quad \mathrm{for}\quad s+1\le i\le n-r+s,$$
    $$D_i=\sum_{j=1}^{n+1-i}e_{i-n+r-s+j-1}e_j^T\quad \mathrm{for}\quad i>n-r+s.$$
    A straightforward calculation shows that the matrices $A_i$ commute, $(A_1,\ldots ,A_n)\in \X{n}{r}{s}$ and that the common kernel of $D_1,\ldots ,D_n$ is trivial. The socle dimension of the algebra is therefore equal to $\dim_{\KK}\bigcap_{i=1}^n\ker A_i=1+\dim_{\KK}\bigcap_{i=1}^n\ker b_i^T=r-s+1$.
\end{example}

\begin{lemma}\label{socle dimension in X}
    Let $1\le s\le r\le n$. The socle dimension of a general algebra corresponding to an $n$-tuple from $\X{n}{r}{s}$ is equal to $r-s+1$.
\end{lemma}

\begin{proof}
    From the description of block matrices and from Remark \ref{3.20} it is clear that the socle dimension is equal to $\dim_{\KK}\bigcap_{i=1}^n\ker A_i\ge r-s+1$.  By Lemma \ref{irreducibility of X} it therefore suffices to find one $n$-tuple where the equality holds, and one such $n$-tuple is in Example \ref{ex: minimal socle dimension in X}.
\end{proof}

Theorem \ref{Gorenstein} then implies the following.

\begin{corollary}\label{s=r}
    Algebras corresponding to $n$-tuples from $\X{n}{r}{r}$ are smoothable if $r\le n$ and $r+n\le 11$.
\end{corollary}

The following is the main result of this subsection.

\begin{proposition}\label{smooth1nr1}
    Let $n\ge 4$, $n+r\le 8$ and $d=n+r+2$. Then each local algebra with Hilbert function $(1,n,r,1)$ belongs to a non-elementary component of $\Hilb_d(\mathbb{A}^n)$.

    Moreover, if such an algebra has socle dimension 2, then it is smoothable.
\end{proposition}
\begin{proof}
    Note that the conditions $n\ge 4$ and $n+r\le 8$ imply $r\le n$. Note also that it suffices to consider algebras supported at zero, and then by the discussion from the beginning of this subsection it suffices to show that $\X{n}{r}{s}$ is a subset of the principal component of $\CommMat{d}{n}$ for each $s=1,\ldots ,r$. Moreover, by Corollary \ref{s=r} it suffices to consider the cases when $1\le s\le r-1$.

    We first consider the case $s=1$. Without loss of generality we may assume that $b_1=e_r$ and $b_i=0$ for $i\ge 2$ (see the proof of Lemma \ref{irreducibility of X}). Commutativity then implies $e_r^TD_i=0$ for $i\ge 2$, so that the $(r+1)$-th row of $A_i$ is zero for $i\ge 2$. It follows that the matrix $E_{r+1,r+1}$ commutes with $A_i$ for each $i\ge 2$. For each $\lambda \ne 0$ the matrix $A_1+\lambda E_{r+1,r+1}$ has two distinct eigenvalues, so the $n$-tuple $(A_1+\lambda E_{r+1,r+1},A_2,\ldots ,A_n)$ belongs to a non-elementary component of $\CommMat{d}{n}$ for each $\lambda\ne 0$, and so does $(A_1,\ldots ,A_n)$.
    
Note that if the socle dimension of an algebra from $\X{n}{r}{1}$ is 2, then $r\le 2$ by Proposition \ref{socle dimension in X}, hence the algebra is smoothable by Propositions \ref{Hf:1n1...1} and \ref{Hf:1n21}.

Next, we consider the case $s=2$. Again we may assume that $f_i=e_i$ for $i=1,\ldots ,n$, $b_1=e_r$, $b_2=e_{r-1}$ and $b_i=0$ for $i\ge 3$. Write $$D_i=\begin{bmatrix}
    D_i^{(1)} & D_i^{(2)} \\ D_i^{(3)} & D_i^{(4)}
\end{bmatrix},$$ where $D_i^{(1)}\in \mathbb{M}_{(r-2)\times 2}$, $D_i^{(2)}\in \mathbb{M}_{(r-2)\times (n-2)}$, $D_i^{(3)}\in \mathbb{M}_2$ and $D_i^{(4)}\in \mathbb{M}_{2\times (n-2)}$. By the irreducibility of $\X{n}{r}{2}$ and Example \ref{ex: minimal socle dimension in X} we may assume that $D_1^{(3)}$ is invertible. Now we use the commutativity conditions. Since $b_1^TD_i=b_2^TD_i=0$ for $i\ge 3$, we get $D_i^{(3)}=0$ and $D_i^{(4)}=0$ for $i\ge 3$. Moreover, from $D_if_j=D_jf_i$ for $i\le 2$ and $j\ge 3$ we get $D_1^{(4)}=D_2^{(4)}=0$. Finally, the equation $b_1^TD_2=b_2^TD_1$ implies $D_2^{(3)}=e_1u^T+e_2e_1^TD_1^{(3)}$ for some $u\in \KK^2$. Write the matrices in the block form with respect to the partition $(1,r-2,2,n,1)$ and define
$$X_1=\begin{bmatrix}
    0 & 0 & 0 & 0 & 0\\
    0 & 0 & 0 & 0 & 0\\
    0 & 0 & \mathbf{I} & 0 & 0\\
    0 & 0 & 0 & 0 & 0\\
    0 & 0 & 0 & 0 & 0
\end{bmatrix}\quad \mathrm{and}\quad X_2=\begin{bmatrix}
    0 & 0 & 0 & 0 & 0\\
    0 & 0 & 0 & 0 & 0\\
    0 & 0 & e_2e_1^T+e_1u^T(D_1^{(3)})^{-1} & 0 & 0\\
    0 & 0 & 0 & 0 & 0\\
    0 & 0 & 0 & 0 & 0
\end{bmatrix}.$$
Then $(A_1+\lambda X_1,A_2+\lambda X_2,A_3,\ldots ,A_n)\in \CommMat{d}{n}$ for each $\lambda \in \KK$. Moreover, for $\lambda \ne 0$ the matrix $A_1+\lambda X_1$ has two distinct eigenvalues, so the $n$-tuple belongs to a non-elementary component of $\CommMat{d}{n}$.

We now show the smoothability statement for algebras corresponding to $n$-tuples from $\X{n}{r}{2}$ that have socle dimension 2. First, by Proposition \ref{socle dimension in X} it follows that $r\le 3$, and by Lemma \ref{irreducibility of X} and Corollary \ref{s=r} we may assume $r=3$. (If $r=2$, we could alternatively use Proposition \ref{Hf:1n21}.) Then $n\in \{4,5\}$. Consider the above $n$-tuple $(A_1+\lambda X_1,A_2+\lambda X_2,A_3,\ldots ,A_n)$ where $\lambda \ne 0$. If $n=4$, then the generalized eigenspace of $A_1+\lambda X_1$ that corresponds to the eigenvalue 0 (i.e., the kernel of $(A_1+\lambda X_1)^9$) has dimension 7, so the irreducibility of $\Hilb_7(\mathbb{A}^4)$ implies that $\cA$ is smoothable. If $n=5$, then we can compute that the restrictions of the matrices $A_1+\lambda X_1,A_2+\lambda X_2,A_3,A_4,A_5$ to the kernel of $(A_1+\lambda X_1)^{10}$ in general determine an algebra with Hilbert function $(1,5,2)$. Such an algebra is smoothable by Theorem \ref{Shafarevich}, hence $\cA$ is smoothable.

It remains to consider the case $n=r=4$, $s=3$. We will find a quadruple in $\X{4}{4}{3}$ which is a smooth point on the principal component of $\CommMat{10}{4}$. The entire irreducible locus $\X{4}{4}{3}$ will then belong to the principal component, and the corresponding algebras will be smoothable.

Let
$$A_1=\begin{bmatrix}
    0 & 0 & 0 & 0 & 1 & 0 & 0 & 0 & 0 & 0\\
    0 & 0 & 0 & 0 & 0 & 0 & 0 & 0 & 1 & 0\\
    0 & 0 & 0 & 0 & 0 & 1 & 0 & 0 & 0 & 0\\
    0 & 0 & 0 & 0 & 0 & 0 & 1 & 0 & 0 & 0\\
    0 & 0 & 0 & 0 & 0 & 0 & 0 & 1 & 0 & 0\\
    0 & 0 & 0 & 0 & 0 & 0 & 0 & 0 & 0 & 1\\
    0 & 0 & 0 & 0 & 0 & 0 & 0 & 0 & 0 & 0\\
    0 & 0 & 0 & 0 & 0 & 0 & 0 & 0 & 0 & 0\\
    0 & 0 & 0 & 0 & 0 & 0 & 0 & 0 & 0 & 0\\
    0 & 0 & 0 & 0 & 0 & 0 & 0 & 0 & 0 & 0
\end{bmatrix}, A_2=\begin{bmatrix}
    0 & 0 & 0 & 1 & 0 & 0 & 0 & 0 & 0 & 0\\
    0 & 0 & 0 & 0 & 0 & 0 & 0 & 1 & -3 & 0\\
    0 & 0 & 0 & 0 & 0 & 0 & 0 & 0 & 0 & 0\\
    0 & 0 & 0 & 0 & 0 & 1 & 0 & 0 & 0 & 0\\
    0 & 0 & 0 & 0 & 0 & 0 & 1 & 0 & 0 & 0\\
    0 & 0 & 0 & 0 & 0 & 0 & 0 & 0 & 0 & 0\\
    0 & 0 & 0 & 0 & 0 & 0 & 0 & 0 & 0 & 1\\
    0 & 0 & 0 & 0 & 0 & 0 & 0 & 0 & 0 & 0\\
    0 & 0 & 0 & 0 & 0 & 0 & 0 & 0 & 0 & 0\\
    0 & 0 & 0 & 0 & 0 & 0 & 0 & 0 & 0 & 0
\end{bmatrix},$$
$$A_3=\begin{bmatrix}
    0 & 0 & 1 & 0 & 0 & 0 & 0 & 0 & 0 & 0\\
    0 & 0 & 0 & 0 & 0 & 0 & 1 & -3 & -1 & 0\\
    0 & 0 & 0 & 0 & 0 & 0 & 0 & 0 & 0 & 0\\
    0 & 0 & 0 & 0 & 0 & 0 & 0 & 0 & 0 & 0\\
    0 & 0 & 0 & 0 & 0 & 1 & 0 & 0 & 0 & 0\\
    0 & 0 & 0 & 0 & 0 & 0 & 0 & 0 & 0 & 0\\
    0 & 0 & 0 & 0 & 0 & 0 & 0 & 0 & 0 & 0\\
    0 & 0 & 0 & 0 & 0 & 0 & 0 & 0 & 0 & 1\\
    0 & 0 & 0 & 0 & 0 & 0 & 0 & 0 & 0 & 0\\
    0 & 0 & 0 & 0 & 0 & 0 & 0 & 0 & 0 & 0
\end{bmatrix}, A_4=\begin{bmatrix}
    0 & 0 & 0 & 0 & 0 & 0 & 0 & 0 & 0 & 0\\
    0 & 0 & 0 & 0 & 0 & 1 & -3 & -1 & -2 & 0\\
    0 & 0 & 0 & 0 & 0 & 0 & 0 & 0 & 0 & 0\\
    0 & 0 & 0 & 0 & 0 & 0 & 0 & 0 & 0 & 0\\
    0 & 0 & 0 & 0 & 0 & 0 & 0 & 0 & 0 & 0\\
    0 & 0 & 0 & 0 & 0 & 0 & 0 & 0 & 0 & 0\\
    0 & 0 & 0 & 0 & 0 & 0 & 0 & 0 & 0 & 0\\
    0 & 0 & 0 & 0 & 0 & 0 & 0 & 0 & 0 & 0\\
    0 & 0 & 0 & 0 & 0 & 0 & 0 & 0 & 0 & 1\\
    0 & 0 & 0 & 0 & 0 & 0 & 0 & 0 & 0 & 0
\end{bmatrix},$$
and $X_1=E_{33}+E_{44}+E_{55}$, $X_2=E_{43}+E_{54}$, $X_3=E_{53}$. Then $(A_1+\lambda X_1,A_2+\lambda X_2,A_3+\lambda X_3,A_4)\in \CommMat{10}{4}$ for each $\lambda \in \KK$. For $\lambda \ne 0$ the matrix $A_1+\lambda X_1$ has two distinct eigenvalues and the generalized eigenspace corresponding to the eigenvalue 0 is 7-dimensional, so the irreducibility of $\Hilb_7(\mathbb{A}^4)$ implies that $(A_1,\ldots ,A_4)$ belongs to the principal component of $\CommMat{10}{4}$. The tangent space to $\CommMat{10}{4}$ at $(A_1,\ldots ,A_4)$ has dimension $130=10^2+3\cdot 10$, so $(A_1,\ldots ,A_4)$ is a smooth point on the principal component, as required.
\end{proof}

\begin{remark}
    Note that in the cases $s\in \{1,2\}$ the above proof is independent of $n$ and $r$.
\end{remark}

\subsubsection{Local algebras with Hilbert function $(1,4,3,2)$}\label{Local1432}

In this subsection, we treat the algebras with Hilbert function $(1,4,3,2)$ and we prove the following main result.
\begin{proposition}\label{HF:1432}
    Local algebras with Hilbert function $(1,4,3,2)$ belong to non-elementary components of $\Hilb_{10}(\mathbb{A}^4)$.

    Moreover, if such an algebra has socle dimension 2, then it is smoothable.
\end{proposition}

\begin{proof}
    We first note that we proved in Proposition \ref{HF1m32graded} that the Hilbert scheme $\HilbFuncGr{(1,4,3,2)}{4}$ has exactly two irreducible components, which we described explicitly. Moreover, \cite[Proposition 4.3]{ccvv} implies that $\HilbFunc{(1,4,3,2)}{4}$ also has exactly two irreducible components, which are the preimages of the two components of $\HilbFuncGr{(1,4,3,2)}{4}$. It therefore suffices to prove the proposition for the general algebras from the two components of $\HilbFunc{(1,4,3,2)}{4}$.

    Let $(\cA,\mathfrak{n})$ be a local algebra with Hilbert function $(1,4,3,2)$ which is supported at 0. Let $(A_1,A_2,A_3,A_4)\in \CommMat{10}{4}$ be the corresponding quadruple of nilpotent commuting matrices. As in the previous subsection we can find a basis such that 
    the matrices $A_i$ are of the form
    \begin{equation}\label{matrices:2341}
    \begin{bmatrix}
        0 & B_i & C_i & 0\\
        0 & 0 & D_i & 0\\
        0 & 0 & 0 & f_i\\
        0 & 0 & 0 & 0
    \end{bmatrix}
    \end{equation}
    with blocks of respective sizes 2, 3, 4, 1, such that $f_1,\ldots ,f_4$ are linearly independent, $\bigcap_{i=1}^4\ker D_i^T=\{0\}$ and $\bigcap_{i=1}^4\ker B_i^T=\{0\}$. Without loss of generality we may assume that $f_i=e_i$ for each $i$.

    If any linear combination of $B_1,\ldots ,B_4$ has rank at most 1, then these matrices either have common 1-dimensional image or common 2-dimensional kernel, see \cite[Lemma 2]{AL}. The first option is not possible, since the common cokernel of the matrices $B_1,\ldots ,B_4$ is trivial. In the second case we may assume without loss of generality that $B_1=E_{13}$, $B_2=E_{23}$ and $B_i=0$ for $i=3,4$. However, then the commutativity implies that the last row of $D_i$ is zero for each $i$, which contradicts the assumption that the common cokernel of the matrices $D_i$ is trivial.
    
    We have shown that some linear combination of $B_1,\ldots ,B_4$ has rank 2. Now we consider two cases.

    {\bf Case 1. Assume that we are in Case (i) of Proposition \ref{HF1m32graded}.} Up to a linear change of variables we assume that $\gr \cA=R/I$ where $I=\Ann(x_1^3,x_2^3,q)$ where $q$ is a quadric such that $q,x_1^2,x_2^2$ are linearly independent. Then $I$ is generated by all quadrics without the monomials $\alpha_1^2,\alpha_2^2$ that annihilate $q$ and all cubics without the monomials $\alpha_1^3,\alpha_2^3$. We may therefore take a $\KK$-basis of $\cA$ of the form $\{\alpha_1^3,\alpha_2^3,q',\alpha_1^2,\alpha_2^2,\alpha_1,\alpha_2,\alpha_3,\alpha_4,1\}$ for some quadric $q'$ without the monomials $\alpha_1^2,\alpha_2^2$, so that the matrices $A_i$ are of the form \eqref{matrices:2341}. Observe that $q'$ is in the kernel of multiplication by any linear form, so the matrices $B_i$ are of the form $B_i=\begin{bmatrix}
        0 & B_i'
    \end{bmatrix}$ for some $2\times 2$ matrices $B_i'$. Furthermore, by a linear change of variables we may assume without any loss of generality that $\mathrm{rank}\, B_1=2$, and then that $B_1'=\mathbf{I}$. Write the matrices $D_i$ in the form $D_i=\begin{bmatrix}
        d_i^T\\D_i'
    \end{bmatrix}$ where $d_i\in \KK^4$ and $D_i'\in \mathbb{M}_{2\times 4}$. Commutativity then implies $D_i'=B_i'D_1'$ for $i=2,3,4$, $(B_i'B_j'-B_j'B_i')D_1'=0$ for $i,j=2,3,4$, and $D_1'e_i=B_i'D_1'e_1$ for $i=2,3,4$.

    First note that $D_1'$ is nonzero, as otherwise the matrices $D_1,D_2,D_3,D_4$ would have a 2-dimensional common cokernel. Assume that $D_1'$ has rank 1. Write $D_1'=uv^T$ for some (nonzero) $u\in \KK^2$ and $v\in \KK^4$. Then $uv^Te_i=B_i'uv^Te_1$ for all $i=2,3,4$. If $v^Te_1=0$, then $uv^Te_i=0$ for $i=2,3,4$, implying either $u=0$ or $v=0$, a contradiction. On the other hand, if $v^Te_1\ne 0$, then $B_i'u$ is a multiple of $u$ for $i=2,3,4$, implying that the matrices $D_1',\ldots ,D_4'$ have a nontrivial common cokernel, which is again a contradiction.

    The above shows that the rank of $D_1'$ is 2. However, the equality $(B_i'B_j'-B_j'B_i')D_1'=0$ then implies that the linear span of $B_1',\ldots ,B_4'$ is a linear space of commuting $2\times 2$ matrices and hence at most 2-dimensional. We may therefore assume that $B_3'=B_4'=0$. Commutativity then implies $D_3'=D_4'=0$, $D_2'=B_2'D_1'$ and $D_1'e_3=D_1'e_4=0$. Now write the matrices in block form with respect to the partition $(2,1,2,4,1)$ and define
    $$X_1=\begin{bmatrix}
        0 & 0 & 0 & 0 & 0\\
        0 & 0 & 0 & 0 & 0\\
        0 & 0 & \mathbf{I} & 0 & 0\\
        0 & 0 & 0 & 0 & 0\\
        0 & 0 & 0 & 0 & 0
    \end{bmatrix},\quad  X_2=\begin{bmatrix}
        0 & 0 & 0 & 0 & 0\\
        0 & 0 & 0 & 0 & 0\\
        0 & 0 & B_2' & 0 & 0\\
        0 & 0 & 0 & 0 & 0\\
        0 & 0 & 0 & 0 & 0
    \end{bmatrix}.$$
    For each $\lambda \in \KK$ the quadruple $(A_1+\lambda X_1,A_2+\lambda X_2,A_3,A_4)$ belongs to $\CommMat{10}{4}$. Since $A_1+\lambda X_1$ has two distinct eigenvalues for $\lambda\ne 0$, the algebra $\cA$ belongs to a non-elementary component of $\Hilb_{10}(\mathbb{A}^4)$. (However, a tangent space computation shows that it does not belong to the smoothable component.)

    For the statement about smoothability observe that the socle of $\cA$ contains $x_1^3$, $x_2^3$ and $q'$, so it is at least 3-dimensional.

    {\bf Case 2. Assume that we are in Cases (ii)-(v) of Proposition \ref{HF1m32graded}.} Cases (iii), (iv), and (v) are degenerations of Case (ii), so we may assume we are in Case (ii). By the proof of Proposition \ref{HF1m32graded} we may therefore assume that $\gr \cA=R/I$ where $I=\Ann (c_1,c_2)$ for some cubics $c_1,c_2$ in the variables $x_1,x_2,x_3$ only, where $c_2$ depends essentially on all three variables.
    It follows that $I$ contains all quadrics divisible by $\alpha_4$. This means that $\alpha_4$ is in the kernel of multiplication by any linear form in $\gr \cA$, or equivalently, that the common kernel of $D_1,\ldots ,D_4$ is nontrivial. We can therefore write $D_i=\begin{bmatrix}
        D_i' & 0
    \end{bmatrix}$ for each $i$. Recall that we assumed that $f_i=e_i$ for each $i$. Commutativity therefore implies $D_4=0$, and then $B_4D_i=0$ for $i=1,2,3$. By the assumption the common cokernel of $D_1,D_2,D_3$ (and $D_4$) is trivial, so we get also $B_4=0$. On the other hand, recall that we already know that some linear combination of $B_1,\ldots ,B_4$ has rank 2, and we may assume that $\rank B_1=2$. 

    Recall that the closure of the locus in $\HilbFunc{(1,4,3,2)}{4}$ that corresponds to Case (ii) of Proposition \ref{HF1m32graded} is irreducible. Therefore we may assume any open condition on the matrices $A_1,\ldots ,A_4$. We therefore assume that some linear combination of $D_1',D_2',D_3'$ is invertible and that the common kernel of $B_1,B_2,B_3$ is trivial. (It is easy to see that some such quadruple $(A_1,\ldots ,A_4)$ exists.) Without loss of generality we may then assume that $D_1'=\mathbf{I}$. Commutativity then implies
    \begin{equation}\label{n^4=0 commutativity}
    D_i'e_1=e_i\,\, \mathrm{for}\,\, i=2,3,\quad D_2'e_3=D_3'e_2, \quad B_i=B_1D_i'\,\, \mathrm{for}\,\, i=2,3\quad \mathrm{and}\quad B_1(D_2'D_3'-D_3'D_2')=0.
    \end{equation}
    Recall that $B_1$ has rank 2. By changing the basis and adding suitable multiples of $A_1$ and $A_2$ to $A_3$ (in order to keep $D_1'=\mathbf{I}$ and $f_3=e_3$) we may assume that $B_1=\begin{bmatrix}
        1 & 0 & 0\\
        0 & 1 & 0
    \end{bmatrix}$. Now after a short computation we see that the equalities \eqref{n^4=0 commutativity} together with the assumption that the common kernel of $B_1,B_2,B_3$ is trivial implies that $D_2'$ and $D_3'$ commute. Let $$X_i=\begin{bmatrix}
        0 & 0 & 0 & 0\\
        0 & D_i' & 0 & 0\\
        0 & 0 & 0 & 0\\
        0 & 0 & 0 & 0
    \end{bmatrix}$$ for $i=1,2,3$. Then $(A_1+\lambda X_1,A_2+\lambda X_2,A_3+\lambda X_3,A_4)\in \CommMat{10}{4}$ for each $\lambda \in \KK$. Since $A_1+\lambda X_1$ has two distinct eigenvalues for $\lambda \ne 0$, the quadruple belongs to a non-elementary component of $\CommMat{10}{4}$.
    
    To prove smoothability observe that the generalized eigenspace of $A_1+\lambda X_1$ corresponding to eigenvalue 0 (i.e. $\ker (A_1+\lambda X_1)^{10}$) has dimension 7, and smoothability then follows from irreducibility of $\Hilb_7(\mathbb{A}^4)$.
\end{proof}

\subsection{Components of $\Hilb_d(\mathbb{A}^n)$ for $d\in \{9,10\}$}\label{subsection:mainTh}

\begin{theorem}\label{mainTh}
    Let $d\in \{9,10\}$. The only elementary components of $\Hilb_d(\mathbb{A}^n)$ are the closures of the loci of local algebras with the Hilbert functions $(1,5,3)$ and $(1,6,3)$.

    Moreover, all local algebras of degree 9 or 10 and socle dimension 2 are smoothable.
\end{theorem}

\begin{proof}
    The closures of the loci of local algebras with Hilbert functions $(1,5,3)$ and $(1,6,3)$ are elementary components of $\Hilb_9(\mathbb{A}^n)$ and $\Hilb_{10}(\mathbb{A}^n)$, respectively, by Theorem \ref{Shafarevich}.

    To show that there are no other elementary components we have to show that each other local algebra belongs to some non-elementary components. By Theorem \ref{11points} it suffices to consider algebras with embedded dimension more than 3. Possible remaining Hilbert functions are therefore $(1,n,r)$ with $r\le 2$, $(1,4,4)$, $(1,5,4)$, $(1,4,5)$, $(1,n,r,1)$ with $n\ge 4$ and $n+r\le 8$, $(1,n,2,2)$ for $n\in\{4,5\}$, $(1,4,3,2)$, $(1,4,2,2,1)$ and the Hilbert functions $(1,n,H_2,\ldots,H_c,1,\ldots,1)$ with at least $c$ $1$'s at the tail. Algebras with these Hilbert functions belong to non-elementary components of $\Hilb_d(\mathbb{A}^n)$ by Theorem \ref{Shafarevich}, Propositions \ref{HF:144,154} and \ref{HF:145}, Proposition \ref{smooth1nr1}, Proposition \ref{Hf:1n22}, Proposition \ref{HF:1432}, Proposition \ref{Sm14221} and Theorem \ref{rayflat}.

    Smoothability in the case of socle dimension 2 follows from the same results, only Corollary \ref{rayflat_socledim2} should be used instead of Theorem \ref{rayflat}.\\
   
\end{proof}

\section{Grassmann cactus variety}\label{GrasCactus}
In this section, as a crucial outcome of previous sections, we establish that certain cactus varieties of pencils are irreducible.\\

Considering the general problem of simultaneous decomposition of a family of polynomials, for a projective variety $X\subset \PP^N$ the $(r,k)-$Grassmann secant variety $\sigma_{r,k}(X)\subset  \Gr(k,\mathbb{K}^{N+1})$ is the closure of the locus of $k$-dimensional subspaces contained in the span of $r$ points.
The case $k=1$ recovers the most studied case, the $r$-the secant variety $\sigma_r(X)$.

On the way to understand the barrier in finding the equations of the secant variety, the $k$-cactus variety is defined as
$$ \mathfrak{K}_r(X)=\overline{\bigcup_{R\subset X} \langle R\rangle}= \overline{\bigcup \lbrace\langle R\rangle\ :\ R\in \Hilb^{Gor}_{\leq r}(X)\rbrace},$$
where the sum is over all finite subschemes $R\subset X$ of length at most $r$. See \cite[Section~2]{JW} for comprehensive details and a proof of the second equality.

A generalization of the Grassmann secant variety is defined as the Grassmann cactus variety $\mathfrak{K}_{r,k}(X)\subset \Gr(r,\KK^{N+1})$, described as
    $$\mathfrak{K}_{r,k}(X)=\overline{\lbrace [V]\in \Gr(k,\KK^{N+1}):\ \PP(V)\subset \langle R\rangle,\ R \subset X,\ \dim H^0(R, \mathcal{O}_R)\leq r \rbrace}.$$
Obviously, $\sigma_{r,k}(X)\subset \mathfrak{K}_{r,k}(X)$. We expect that the Grassmann cactus variety $\mathfrak{K}_{r,k}(X)$ exhibits similar, though weaker, ``barrier'' properties as the standard cactus variety. \\

The key point is that $\mathfrak{K}_{r,k}$ depends only on the algebras of socle dimension at most $k$.
\begin{definition}
  Let $r$ and $k$ be positive integers, and let $X$ be a variety. We define
  \begin{align*}
      \Hilb^{\leq k}_{\leq r}(X) = \{ [R] \in \Hilb_{\leq r}(X):\ & R = \Spec A,\ A \cong A_1\times\cdots\times  A_m,
      \\ &\text{ every } A_i \text{ is a local algebra of socle dimension at most } k \}.
  \end{align*}
  It is an open subscheme of $\Hilb_{\leq r}$. By \cite[Theorem 1]{bbg} or \cite[Proposition 4.2]{JK22}  we have:
\end{definition}
\begin{theorem}
\label{theorem:cactus_vs_hilbert}
If $X$ is a smooth projective variety, then    
\begin{equation*}
    \mathfrak{K}_{r,k}(X) = \overline{ \{ [V] \in \Gr(k,\KK^{N+1}) : \PP V \subset \langle R \rangle,\ R \in \Hilb^{\leq k}_{\leq r}(X) \}}.
\end{equation*}
\end{theorem}

By the above description, using the language of Hilbert schemes, the problem of finding the smallest $r$ where $\sigma_{r,k}\neq \mathfrak{K}_{r,k}$ translates into the problem of finding the smallest $r$ such that the locus of subschemes of length $r$ and socle dimension at most $k$, denoted by $\Hilb^{ \leq k}_{r}(\mathbb{A}^n)$, is not contained in the smoothable component $\HilbFuncSm{r}{n}$. \\

By the results of \cite{JW}, it is known that for $r\leq 13$ and $X$ smooth projective, the cactus and secant variety are identical. The first extremal case $\mathfrak{K}_{14}(\nu_d(\PP^n))$ is studied in \cite[Section~5]{gmr} where it is shown that $\mathfrak{K}_{14}(\nu_d(\PP^n))$ has one more component than $\sigma_{14}(\nu_d(\PP^n))$, for $ n,d\geq 6$. Moreover, from the results of \cite{ccvv} and of \cite[Section~6]{gmr}, it follows that $r=8$ and $n\geq 4$ is the first case where $\sigma_{8,3}(\nu_d(\PP^n))\neq \mathfrak{K}_{8,3}(\nu_d(\PP^n))$.\\

The intermediate case $k=2$ is not known. By \cite{JW}, we know that all Gorenstein (socle dimension $1$) algebras of degree $r\leq 13$ are smoothable. In \cite{JK22}, the authors prove that $\Hilb^{\leq 2}_r(\mathbb{A}^n,0)$, parametrizing all local subschemes of length $r$ and socle dimension at most $2$ supported at the origin, is of dimension bounded from above by $(r-1)(n-1)$ for $r\leq 11$, though they believe this bound also holds true for $r=12$. On the other hand, there are examples of non-smoothable algebras of degree $r=13$  and socle dimension $2$ in hand, given by Hilbert function $(1,5,6,1)$. Moreover, by results of \cite{ccvv}, any local algebra of length at most $8$ and socle dimension $2$ is smoothable. 
Therefore,  we can look for an example of $\sigma_{r,2}(\nu_d(\PP^n))\neq \mathfrak{K}_{r,2}(\nu_d(\PP^n))$  among $r$ with $9\leq r\leq 13$.

As a consequence of Theorem  \ref{mainTh} showing that all local algebras of degree $r=9,10$ and socle dimension $2$ are smoothable, the following result improves this bound to $11\leq r\leq 13$.

 \begin{theorem}
    For $r=9,10$ and $n\geq 4$ and any $d \geq 1$, the Grassmann cactus variety $\mathfrak{K}_{r,2}(\nu_d(\PP^n))$  is irreducible and coincides with $\sigma_{r,2}(\nu_d(\PP^n))$.
\end{theorem}
This is a particular case of the following theorem.
\begin{theorem}
   Let $X \subset \PP^N$ be a projective smooth non-degenerate variety. For $r = 9,10$, the Grassmann cactus variety $\mathfrak{K}_{r,2}(X)$ is irreducible and coincides with $\sigma_{r,2}(X)$.
\end{theorem}
\begin{proof}
  Let $n = \dim X$. The variety $X$ is \'etale-locally isomorphic to the affine space $\AA^n$, hence there is a bijection of the irreducible components of $\Hilb^{\leq 2}_{\leq r}(X)$ and $\Hilb^{\leq 2}_{\leq r}(\AA^n)$. For $r = 9, 10$, the scheme $\Hilb^{\leq 2}_{\leq r}(\AA^n)$ is irreducible, hence we get that $\Hilb^{\leq 2}_{\leq r}(X)$ is irreducible, and this means that the smoothable component is the only component. Therefore, by Theorem~\ref{theorem:cactus_vs_hilbert}, we conclude that $\mathfrak{K}_{r,2}(X)$ is irreducible and is equal to $\sigma_{r,2}(X)$.
\end{proof}
  {\small
\bibliographystyle{alpha}
\bibliography{biblio}}
\end{document}